\documentclass{amsart}

\usepackage{verbatim}
\usepackage{enumitem}
\usepackage{xcolor}
\usepackage{tikz}
\usepackage{tikz-cd}
\usepackage{arydshln}
\usepackage{caption}
\usepackage{subcaption}
\usepackage{graphicx}
\usepackage{ulem}
\tikzset{
	symbol/.style={
		draw=none,
		every to/.append style={
			edge node={node [sloped, allow upside down, auto=false]{$#1$}}}
	}
}

\usepackage{amsmath}
\usepackage{bbm}
\usepackage{hyperref}
\usepackage{times,amsfonts,amsmath,amstext,amsbsy,amssymb,
amsopn,amsthm,upref,eucal}
\usepackage[T1]{fontenc}

\newtheorem{theorem}{Theorem}[section]
\newtheorem{proposition}[theorem]{Proposition}
\newtheorem{lemma}[theorem]{Lemma}
\newtheorem{corollary}[theorem]{Corollary}

\theoremstyle{definition}
\newtheorem{definition}[theorem]{Definition}

\theoremstyle{remark}
\newtheorem{remark}[theorem]{Remark}

\newtheorem{question}[theorem]{Question}

\numberwithin{equation}{section}
\newcommand{\wX}{\widehat{X}}
\newcommand{\wY}{\widehat{Y}}
\newcommand{\wW}{\widehat{W}}
\newcommand{\wR}{\widehat{R}}
\newcommand{\wT}{\widehat{T}}

\newcommand{\field}[1]{\mathbb{#1}}
\newcommand{\R}{\field{R}}
\newcommand{\N}{\field{N}}
\newcommand{\C}{\field{C}}
\newcommand{\Z}{\field{Z}}

\newcommand{\Om}{\Omega}
\newcommand{\minuszero}{\backslash\{0\}}

\begin{document}

\title[Heisenberg translation flows]{Heisenberg translation flows}


%    author one information
%\author{Francisco Arana--Herrera}
%\address{Department of Mathematics, Stanford University, 450 Serra Mall
%	Building 380, Stanford, CA 94305-2125, USA}
%\curraddr{}
%\email{farana@stanford.edu}
%\thanks{}

\author{Francisco Arana--Herrera}

\email{fa50@rice.edu}

\address{Department of Mathematics, Rice University,
Herman Brown Hall for Mathematical Sciences
6100 Main Street
Houston, TX 77005.}

\author{Jayadev Athreya}

\email{jathreya@uw.edu}

\address{Department of Mathematics, Padelford Hall, 4110 E Stevens Way NE,  University of Washington, Seattle, WA, USA.}
	
\author{Giovanni Forni}

\email{gforni@umd.edu}

\address{Department of Mathematics, William E. Kirwan Hall, 4176 Campus Dr, University of Maryland, College Park, MD 20742, USA.}

%    author two information
%\author{}
%\address{}
%\curraddr{}
%\email{}
%\thanks{}

%\subjclass[2000]{Primary }
%    The 2010 edition of the Mathematics Subject Classification is
%    now available.  If you are citing a classification from the
%    new scheme, use the following input coding instead.
%\subjclass[2010]{Primary }

\date{}

\begin{abstract}
    We study ergodic theoretical properties of flows on circle bundles over translation surfaces that arise via prequantization, generalizing the theory of Heisenberg nilflows to base surfaces more general than tori; these flows are among the most fundamental examples of parabolic dynamical systems with non-trivial central directions. In particular, we show that such flows are relatively mixing, i.e., they exhibit decay of correlations in the orthogonal complement of functions constant along fibers. We discuss applications of this result to the dynamics of such flows, to the ergodic theory on the corresponding space of wave functions, and, via surface of section constructions, to the study of affine skew products over interval exchange transformations, in the spirit of Furstenberg's classification program for measurable dynamical systems. 
\end{abstract}

\maketitle


\thispagestyle{empty}

\tableofcontents

\section{Introduction}\label{sec:intro}

\subsection{Motivation.}\label{sec:motivation} 

Quantum mechanics models phenomena sharing three important characteristics: discretization of energy levels, interference of physical states, and uncertainty of measurable quantities. The first two properties can be attained by considering the spectrum of unbounded self-adjoint linear operators on a Hilbert space of wave functions, usually the space of square-integrable sections of a Hermitian line bundle on the underlying physical space. Uncertainty, on the other hand, requires joining the fibers of such line bundle using a non-flat connection. Such models are commonly referred to as prequantizations. Physical interpretations of such constructions that capture the leading terms of the spectrum, usually referred to as a quantizations, require further assumptions on the sections considered, e.g., assuming they are holomorphic, to reduce to finite dimensional  analysis. \\

\paragraph*{\bf The Schr\"odinger equation} In (pre)quantizations of Hamiltonian systems, the time evolution of wave functions is commonly described by means of the Schr\"odinger equation:
\begin{equation}\label{eq:schrodinger}
i\hbar \frac{d}{dt} | \psi(t) \rangle = \widehat{H}|\psi(t) \rangle,
\end{equation}
where $\hbar$ is the wavelength or Planck constant of the system, $\psi(t)$ is the time evolution of the wave function, and $\smash{\widehat{H}}$ is a self-adjoint operator corresponding to the underlying Hamiltonian. The time evolution $\psi(t)$ can be described as the compound of two effects, the \textit{geometric phase}, which transports the wave function along the orbits of the Hamiltonian flow on physical space, and the \textit{dynamical phase}, which spreads out the wave function over physical space in terms of the underlying Hamiltonian.\\

\paragraph*{\bf Parabolic dynamics} Many dynamical systems of interest are not Hamiltonian. This observation is particularly important when considering models of slow chaos, i.e., parabolic dynamics, such as linear flows on translation surfaces, which are, moreover, generically minimal. Although not Hamiltonian, these flows are locally Hamiltonian, or, equivalently in the case of surfaces, area-preserving. 

The main goal of this paper is to introduce and study models of quantum mechanics where the corresponding \textit{geometric phase} is driven by a linear flow on a general translation surface $(M, \omega)$. We are interested in how the slow  chaotic behavior of these systems gives rise to the spreading out phenomena usually guaranteed by the \textit{dynamical phase}. We emphasize that, as the dynamical systems we consider are not globally Hamiltonian, but only locally Hamiltonian, this cannot be achieved by quantizing a Hamiltonian,  as in the case of integrable dynamical systems, i.e., tori~\cite{Faure07}.

Analogous models in the regime of fast chaos, i.e., over hyperbolic dynamical systems, have been studied by Faure and Tsuji \cite{FT15}, but to the knowledge of the authors, this paper constitutes the first study of such models in the parabolic case. We also note that, to the knowledge of the authors, the flows we introduce and study in this paper are the first examples of non-homogeneous parabolic dynamical systems with a non-trivial central direction. It is also important to highlight that, from an axiomatic point of view, the models we describe are indeed prequantizations of the underlying parabolic flow on physical space, even before passing to the semiclassical limit $\hbar \to 0$.\\

\paragraph*{\bf Line and circle bundles} For the models we consider, the time evolution of wave functions is induced by the dynamics of parallel transport on a circle bundle $C$ corresponding to an underlying Hermitian line bundle $L$ over the surface $(M, \omega)$ with prescribed non-trivial curvature conditions. These flows are, in their own right, interesting examples of parabolic flows in three dimensions with a central direction, the fiber direction, which is not tangent to the flow direction. These generalize Heisenberg nilflows to base surfaces more general than tori; hence we refer to the circle bundle as a \textit{Heisenberg translation bundle} and to the flow as a \textit{Heisenberg translation flow}. Furthermore, their first return maps to natural cross sections correspond to skew products over interval exchange transformations induced by explicit piecewise affine maps to the circle; this is a subject of great interest following Furstenberg's classification program for measurable dynamical systems.\\

\paragraph*{\bf Main results} The main results of this paper are the following:
\begin{enumerate}
    \item \textit{Relative mixing:} Heisenberg translation flows are relatively mixing, i.e., they are mixing on the orthogonal complement of functions constant along fibers; see Theorem \ref{theo:real_weak_mix}.
    \item \textit{Decay of correlations:} Heisenberg translation flows on wave functions exhibit decay of correlations, with speed depending on the regularity of the functions; see Theorem~\ref{thm:rel_mix_new} and Corollary \ref{cor:decay}. %\giovanni{The speed of decay depends on the regularity of the function. In any case Corollary \ref{cor:decay} gives no speed}
    \item \textit{Spectra:} The (continuous) spectrum of Heisenberg translation flows is all of $\mathbb{R}$; see Proposition \ref{prop:spectra}.
    \item \textit{Affine skew products:} First return maps of Heisenberg translation flows to natural cross sections correspond to skew products over interval exchange transformations induced by explicit piecewise affine maps to the circle. See Theorem \ref{theorem:admissible}.\\
\end{enumerate}

Below we elaborate more on each of these results.

\subsection{Relative mixing.}\label{sec:relative} A fundamental question in this setting is whether the dynamical and ergodic theoretical properties of the base translation surface can be exported to the lifted flows on the corresponding Heisenberg translation bundles. A unified treatment of these ideas can be achieved via the concept of \textit{relative mixing}, i.e., decay of correlations for Lebesgue $L^2$ eigenfunctions on the orthogonal complement of the subspace of functions constant along fibers. Our main result in this context is the following; for a more precise quantitative statement see Theorem \ref{thm:rel_mix_new}:

\begin{theorem}
    \label{theo:real_weak_mix}
    The lift of any translation flow on any translation surfaces to any Heisenberg translation bundle is relatively  mixing.
\end{theorem}

As a direct consequence of Theorem \ref{theo:real_weak_mix} we deduce the following corollary; for a more precise statement see Corollary \ref{cor:ethry}:

\begin{corollary}
    \label{cor:export}
    Suppose a translation flow on a translation surface is minimal, ergodic with respect to the Lebesgue measure, uniquely ergodic, or weak mixing with respect to the Lebesgue measure. Then, the lift of such a flow to any Heisenberg translation bundle has the corresponding dynamical properties.
\end{corollary}

\subsection{Prequantization.}\label{subsec:prequant} Classical observables of a translation surface $(M,\omega)$ correspond to elements of $L^2(M, \omega)$, the space of real-valued square-integrable functions on $M$ with respect to the area form induced by $\omega$. \textit{Prequantum observables} correspond to unbounded self-adjoint operators on $\Gamma^2(L)$, the Hilbert space of $L^2$ sections of the chosen Hermitian line bundle $L$; denote the group of such operators by $\mathfrak{u}(\Gamma^2(L))$. Prequantization seeks an assignment between classic and quantum observables, that is, a map \[
    \mathcal{O} \colon L^2(M) \to \mathfrak{u}(\Gamma^2(L))
    \] satisfying \textit{Dirac's quantum conditions}; see Definition~\ref{def:prequant}. We will define a \textit{geometric prequantization} satisfying such conditions; see Definition~\ref{def:geomprequant}. For this construction it will be crucial that the Hermitian line bundles we consider have non-vanishing curvature. 

\subsection{Dynamics on wave functions.}\label{subsec:wavedyn} Our goal is not to simply prequantize the observables of a translation surface, but also to prequantize the dynamics of the corresponding linear flows. A \textit{prequantization} of a linear flow $\{\phi_t^X \colon M \to M\}_{t \in \mathbb{R}}$ on a translation surface $(M,\omega)$ with respect to a prequantization $\mathcal{O} \colon L^2(M) \to \mathfrak{u}(\Gamma^2(L))$ is a one-parameter subgroup $\mathcal{U} := \{\mathcal{U}_t\}_{t \in \mathbb{R}} \subseteq U(\Gamma^2(L))$ of the group $U(\Gamma^2(L))$ of bounded unitary operators on $\Gamma^2(L)$ satisfying the following conjugacy relation for $f$ in the domain of $\mathcal O$:
$$(\mathcal{U}^X_t)^{-1} \circ \mathcal{O}(f) \circ \mathcal{U}^X_t = \mathcal{O}(f \circ \phi^X_t);$$ see Definition~\ref{def:prequant:dynamics}. We will introduce a family of \textit{wave transforms} that will map the space $\Gamma^2(L^k)$ of square-integrable sections of the $k^{th}$ tensor power $L^k$ of the line bundle $L$ to a subspace $E_{-k} \subset L^2(C)$, in fact, an eigenspace of an appropriate central operator, in such a way that the natural flows arising on each space are intertwined; see Theorem~\ref{theo:iso} for a precise statement. This will allow us to derive several consequences to Theorem \ref{theo:real_weak_mix} in the prequantized setting. In particular, we will deduce spectral properties of $\nabla_X$, the corresponding prequantized Schrödinger operator.

\subsection{Affine skew products over interval exchange transformations.}\label{sec:affineskew} The first return map of any translation flow to any transverse geodesic segment of any translation surface is an interval exchange transformation (IET) with piecewise constant first return times, colloquially referred to as the \textit{heights}. Reciprocally, as illustrated by Veech's famous \textit{zippered rectangles} construction \cite{veech}, every interval exchange transformation can be suspended using suitable piecewise constant roof functions to obtain translation surfaces. Given an IET $T$ of $d$ intervals, the set of suitable piecewise constant roof functions is a cone in $\R_+^d$, which we denote $H_T$. An analogous, albeit more complicated, picture also holds for Heisenberg translation bundles. Using the connection of the bundle we can parallelize it over any geodesic segment of the base translation surface. The first return map to the corresponding cross section is then a piecewise affine skew product over the underlying interval exchange transformation. More precisely, we show the following; see Theorem \ref{theo:skew2} for a more explicit version of this result:

\begin{theorem}
    \label{theorem:admissible}
    Let $T \colon I \to I$ be an IET of $d$ intervals. Denote $I = \cup_{j=1}^d I_j$ with the $I_j$ being the intervals of continuity of $T$. Then, there exists an explicit affine subspace $\mathcal B_T \subseteq \R^d$ such that the set $\mathcal{A}_T = H_T \times \mathcal B_T \subset \R_+^d \times \R^d$ parameterizes the  piecewise-linear skewing functions, with pieces corresponding to the intervals $I_j$, such that the skew-products over $T$ with such skewing functions are exactly those which arise as first return maps to the corresponding cross sections  of the lift to Heisenberg translation bundles of piecewise constant suspension flows over $T$. More precisely, the first return maps $\smash{\wT}: I \times \R/\Z \to I \times \R/\Z$ which occur are of the form: $$\wT(x, \rho) = (T(x), \rho + h_j x+b_j) \quad \text{if} \quad x \in I_j, \quad \text{for $(h, b) \in \mathcal A_T$}.$$
\end{theorem}

\paragraph*{\bf Dynamical properties} As a direct consequence of Theorem \ref{theorem:admissible} and Corollary \ref{cor:export} we deduce the following result:

\begin{corollary}\label{cor:admissible}
    Let $T \colon I \to I$ be an IET. Suppose $T$ is minimal, ergodic with respect to the Lebesgue measure, or uniquely ergodic. Then, any skew-product over $T$ given by a skewing function $a = (h, b) \in \mathcal{A}_T$ has the corresponding dynamical properties.
\end{corollary}

\paragraph*{\bf The cohomological equation} Furthermore, using an argument of Furstenberg~\cite{Fu61}, we will deduce the following result from Corollary~\ref{cor:admissible}:

\begin{corollary}\label{cor:cohomological} Let $T: I \rightarrow I$ be a uniquely ergodic IET and $g = (h, b) \in \mathcal A_T$ be an admissible skewing function. Then there is no $L^2$-solution $u$ to the cohomological equation \begin{equation}\label{eq:cohom} u\circ T - u = g.\end{equation}
\end{corollary}

\paragraph*{\bf Organization} The paper is organized as follows: in \S\ref{sec:heisenberg} we introduce the class of Hermitian line bundles and associated circle bundles, referred to as \textit{Heisenberg translation bundles}, on which we will do dynamics. In \S\ref{sec:heisenberg:flow} we introduce the notion of \textit{Heisenberg translation flows} and prove our main result on them, Theorem \ref{theo:real_weak_mix}. In \S\ref{sec:spectral} we state and prove our results on the spectra of the associated Koopman operators of these flows. In \S\ref{sec:prequant} we formalize our notion of prequantization. In \S\ref{sec:dynwave} we state and prove our results on the corresponding wave functions and their decay of correlations. Finally, in \S\ref{sec:affineskewiet} we prove Theorem~\ref{theorem:admissible} and connect this to ideas of Furstenberg to prove Corollary~\ref{cor:cohomological}. \\

\paragraph*{\bf Acknowledgments} We thank Artur Avila for suggesting the proof of Lemma~\ref{lemma:no_gaps}. We also would like to thank Mihaljo Ceki\'c, Erwan Lanneau, and Gabriel Paternain for useful discussions.  JSA was partially supported by the United States National Science Foundation (DMS 2404705), the Pacific Institute for the Mathematical Sciences, and the Victor Klee Faculty Fellowship at the University of Washington. GF was partially supported by by the United States National Science Foundation (DMS 2154208) and co-funded by the European Union (ERC-AdG 101142840-BRen). Views and opinions expressed are however those of the authors only and do not necessarily reflect those of the European Union or the European Research Council. Neither the European Union nor the granting authority can be held responsible for them.

\section{Heisenberg translation bundles}\label{sec:heisenberg}

We now introduce the main objects of study of this paper, \textit{Heisenberg translation bundles}, which arise as geometric prequantizations of translation surfaces. Following \cite{quant}, we give a brief overview of the general theory as well as the basic existence and classification results in our context.

\subsection{Heisenberg translation bundles.}\label{sec:heisenberg:bundles} Let $(M,\omega)$ be a translation surface, i.e., $M$ is a Riemann surface and $\omega$ is a holomorphic $1$-form, also called an Abelian differential, on $M$. The differential $\omega$ induces a singular symplectic form $\eta := \frac{i}{2}(\omega \wedge\overline{\omega})$ on $M$; the singularities of the symplectic form correspond to the zeroes of the differential. In coordiantes where $\omega = dz = dx+idy$ we have $\eta = dx \wedge dy$. The total area of $(M,\omega)$ is
\[
\mathrm{Area}(M,\omega):= \int_M \eta.
\]
We will also use $\eta$ to refer to the Euclidean measure on $M$ induced by the form $\eta$.
\begin{definition}
    \label{def:hcb}
    Let $(M,\omega)$ be a translation surface, $\eta := \frac{i}{2}(\omega \wedge\overline{\omega})$ be its corresponding symplectic form, and $\hbar > 0$. A \textit{Heisenberg line bundle} on $(M,\omega)$ of wavelength $\hbar$ is a Hermitian line bundle $L \to M$ endowed with a connection $\nabla$ compatible with the Hermitian structure and of curvature $\hbar^{-1}\eta$. Restricting this line bundle to elements of Hermitian norm $1$ yields a circle bundle $C \to X$ we refer to as a \textit{Heisenberg circle bundle}. We write $L_m$ and $C_m$ for the fibers of $L$ and $C$, respectively, over $m \in M$. 
\end{definition}

\begin{remark}
    The reader unfamiliar with the vocabulary of differential geometry can interpret Definition \ref{def:hcb} as follows. The fibers of the line bundle $L \to M$ can be communicated using a parallel transport operation along paths on the base induced by the connection $\nabla$; see Figure \ref{fig:curv}. This operation preserves the Hermitian inner product on fibers. In particular, to every closed loop $\gamma$ on $M$ one can assign a complex number $\mathrm{hol}(\gamma) \in \mathbb{C}$ of modulus $1$ representing the effect of such parallel transport along the loop. The curvature condition can be translated as follows: If $\gamma$ is a loop bounding a disk $D \subseteq M$ then:
    
    \begin{equation}
    \label{eq:curv}
    \mathrm{hol}(\gamma) = \exp\left(i \int_D \eta \right).
    \end{equation}
\end{remark}

	\begin{figure}[ht]
		\centering
		\begin{subfigure}{.45\textwidth}
			\centering
			\vspace{+0.5cm}
			\includegraphics[scale=.8]{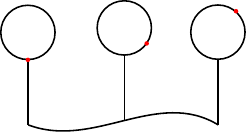}
		\end{subfigure}
		\begin{subfigure}{.45\textwidth}
			\centering
			\vspace{+.64cm}
			\includegraphics[scale=.8]{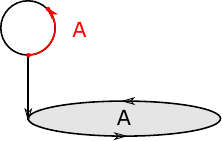}
		\end{subfigure}
        \caption{Parallel transport and the curvature relation \eqref{eq:curv}.}
        \label{fig:curv}
	\end{figure}

\begin{remark}
    For the rest of the paper, whenever a connection is considered on a Hermitian line bundle, we will always assume it is compatible with the Hermitian structure.
\end{remark}

\paragraph*{\bf Hermitian line bundles on tori and Heisenberg nilmanifolds} Given a Hermitian line bundle $L$ on the torus $\R^2/\Z^2$, viewed as the translation surface $(\C/\Z[i], dz)$, together with a connection with curvature corresponding to the natural associated symplectic form, the associated circle bundle $C$ will be a Heisenberg nilmanifold, that is, a quotient of the Heisenberg group     
    \[
    H := \left\lbrace \left(\begin{array}{ccc}
        1 & x & r \\
        0 & 1 & y \\
        0 & 0 & 1
    \end{array} \right) \ \colon\ p,q,r \in \mathbb{R} \right\rbrace
    \]
    by a lattice $\Lambda$ whose covolume corresponds to the wavelength. The three natural vector fields $\smash{\widehat{X}, \widehat{Y}, \widehat{R}}$ associated to the coordinates $x$, $y$, $r$ correspond to the horizontal, vertical, and fiber direction flows on the circle bundle $C$ and satisfy the Heisenberg relations \begin{equation}\label{eq:heisrel} \left[\widehat{X}, \widehat{Y}\right] = \widehat{R}, \quad \left[\widehat{X}, \widehat{R}\right] = 0,  \quad \left[\widehat{Y}, \widehat{R}\right] = 0.
    \end{equation}

\subsection{Existence and classification.} A first natural question is whether Heisenberg line bundles exist over any (higher genus) translation surface and, if so, how to classify such bundles. The answer to these questions depends on the \textit{Weil integrality condition} which in our context translates to
\[
\mathrm{Area}(M,\omega) = \int_{M} \eta\in 2\pi\mathbb{Z}.
\]
Because $M$ is a surface, this condition is equivalent to
\[
(2\pi)^{-1}[\eta] \in H^2(M;\mathbb{Z}) \subseteq H^2(M;\mathbb{R}),
\]
which is the more traditional way this condition can be found in the literature.\\

\paragraph*{\bf Contractible open covers} To state and prove our main technical lemma, the \textit{Weil integrality lemma} (see Lemma~\ref{lem:weil}), which is needed to establish the existence of Heisenberg line bundles, we first recall that an open covering $\mathcal{U}$ of a surface $M$ is said to be \textit{contractible} if
\begin{enumerate}
    \item $\mathcal{U}$ is locally finite.
    \item Every $U \in \mathcal{U}$ is relatively compact.
    \item Every non-empty finite intersection of elements of $\mathcal{U}$ is contractible.
\end{enumerate}
\begin{remark}\label{rem:opencovers} It is well known that every surface admits a contractible open cover; this can be done, for instance, by considering a triangulation. Furthermore, the set of contractible open covers is cofinal in the partially ordered set of open covers under refinement. 
\end{remark}

\paragraph*{\bf Weil integrality} We are now ready to state Lemma~\ref{lem:weil}, the \textit{Weil integrality lemma}; we do not present a proof here but the result follows by using the standard isomorphism between de Rham cohomology and \v{C}ech cohomology; see for instance \cite[Appendix A.6]{quant}. %We remark that the constant $\hbar > 0$ in the statement of the lemma can be interpreted as the wavelength of the physical model under consideration.

\begin{lemma}
    \label{lem:weil}
    Let $M$ be a surface, $\eta$ be a $2$-form on $M$ with $$(2\pi\hbar)^{-1}[\eta] \in H^2(M;\mathbb{Z}) \subseteq H^2(M;\mathbb{R})$$ for some $\hbar > 0$, and $\mathcal{U}$ be a contractible open cover of $M$. Then, there exist 
    \begin{itemize}
        \item $1$-forms $\theta_U$ defined on $U$ for every $U \in \mathcal{U}$,
        \item real valued functions $\varphi_{UV}$ defined on $U \cap V \neq \emptyset$ for every non-empty intersection $U\cap V \neq \emptyset$ of elements $U,V \in \mathcal{U}$,
    \end{itemize}
     such that the following conditions hold:
    \begin{enumerate}
        \item For every $U \in \mathcal{U}$,
        \[
        d\theta_U = \eta \quad \text{on} \quad U.
        \]
        \item For every non-empty intersection $U\cap V \neq \emptyset$ of elements $U,V \in \mathcal{U}$,
        \[
        d\varphi_{UV} = \theta_U  - \theta_V \quad \text{on} \quad U\cap V.
        \]
        \item For every non-empty intersection $U \cap V \cap W \neq \emptyset$ of elements $U,V,W \in \mathcal{U}$,
        \[
        f_{UVW} := (\varphi_{UV} + \varphi_{VW} + \varphi_{WU})/2\pi\hbar
        \]
        is constant and integer valued on $U \cap V \cap W \neq \emptyset$.
    \end{enumerate}
\end{lemma}

\paragraph*{\bf Existence and classification} We will introduce the existence and classification results for Heisenberg line bundles in the same statement. The classification part of the result relies on the classification of Hermitian line bundles with a compatible flat connection, i.e., of curvature zero; such a classification can be obtained by recalling that flat Hermitian line bundles are completely specified by their holonomy representations. More concretely:

\begin{theorem}
    \label{theo:flat}
    Let $M$ be a surface. Then, there exists a $1$-to-$1$ correspondence
    \[
    \left\lbrace \begin{array}{c}
         \text{Isomorphism classes of } \\
         \text{line bundles on $M$ with a flat connection}
    \end{array}
    \right\rbrace \longleftrightarrow H^1(M;\mathbb{R})/H^1(M;\mathbb{Z}).
    \]
\end{theorem}

We are now ready to introduce the main result of this section; we follow the arguments in \cite[Proposition 8.3.1]{quant} but allow for singularities on the symplectic form considered.

\begin{theorem}
\label{theo:prequant}
Let $(M,\omega)$ be a translation surface and $\eta := \frac{i}{2}(\omega \wedge\overline{\omega})$ be its corresponding symplectic form. Then, the following statements are equivalent:
\begin{enumerate}
    \item There exists a Hermitian line bundle on $M$ with a connection of curvature $\hbar^{-1} \eta$.
    \item Weil's integrality condition holds, i.e.,
    \begin{equation}\label{eq:weil}
    (2\pi\hbar)^{-1}[\eta] \in H^2(M;\mathbb{Z}) \subseteq H^2(M;\mathbb{R}).
    \end{equation}
\end{enumerate}
Furthermore, when any of these conditions hold, tensoring with (isomorphism classes of) Hermitian line bundles on $M$ with a flat connection gives a $1$-to-$1$ correspondence
 \[
    \left\lbrace \begin{array}{c}
         \text{Isomorphism classes of} \\
         \text{Hermitian line bundles on $M$} \\
         \text{with a connection of curvature $\hbar^{-1} \eta$}
    \end{array}
    \right\rbrace \longleftrightarrow H^1(M;\mathbb{R})/H^1(M;\mathbb{Z}).
\]
\end{theorem}

\begin{remark}\label{rem:affine}
    The $1$-to-$1$ correspondence in Theorem \ref{theo:prequant} should not be thought of as a linear identification but rather as an affine one. This general theme will be present throughout the rest of the paper; see for instance Theorem \ref{theo:skew2}.
\end{remark}

\begin{proof}
    Suppose first that Weil's integrality condition holds. Consider a contractible open cover $\mathcal{U}$ of $M$. Let $\theta_U$ and $\varphi_{UV}$ be the $1$-forms and functions provided by Lemma \ref{lem:weil} for the closed $1$-form $\eta$ on $M$. For every non-empty intersection $U \cap V \neq \emptyset$ of elements $U,V \in \mathcal{U}$ consider the complex valued functions of modulus $1$ given by
    \[
    c_{UV} := \exp(i\varphi_{UV}/\hbar).
    \]
    Condition (3) in Lemma \ref{lem:weil} precisely guarantees that, for every non-empty intersection $U\cap V\cap W \neq \emptyset$ of elements $U,V,W \in \mathcal{U}$,
    \[
    c_{UV} c_{VW} c_{WU} = \exp(2\pi i (\varphi_{UV} + \varphi_{VW} + \varphi_{WU})/2\pi\hbar) = 1.
    \]
    In particular, the functions $c_{UV}$ are the transition functions of a complex line bundle $L \to M$, which, furthermore, because the functions $c_{UV}$ have modulus $1$, admits a Hermitian structure. Let $U \times \mathbb{C}$ be a local trivialization of $L$ over $U \in \mathcal{U}$. Denote by $s \colon U \to L$ the local section of $L$ on $U$ corresponding to fixing $1 \in \mathbb{C}$ as the second coordinate of the local trivialization $U \times \mathbb{C}$. Then, a connection $\nabla$ of $L$ on $U$ is completely determined by a complex valued $1$-form $\theta$ on $U$ through the equation
    \begin{equation}
    \label{eq:con}
    \nabla_\cdot s = -i \theta(\cdot) s.
    \end{equation}
    In this local representation, the curvature $2$-form of $\nabla$ on $U$ is given by $d\theta$. For each $U \in \mathcal{U}$, let $\theta$ in \eqref{eq:con} be $\theta_U/\hbar$. The condition that these local connections glue to a global connection on $L$ is precisely given by
    \begin{equation}
    \label{eq:glue}
    dc_{UV}/c_{UV} = i(\theta_U - \theta_V)/\hbar, \quad \text{for every } U,V \in \mathcal{U} \text{ with } U\cap V \neq\emptyset.
    \end{equation}
    This condition can be checked directly from condition (2) in Lemma \ref{lem:weil}. Thus, we obtain the connection $\nabla$ on $L$. Condition (1) in Lemma \ref{lem:weil} guarantees that the curvature of this connection is $\hbar^{-1} \eta$. Finally, the fact that the $1$-forms $\theta_U$ are real valued guarantees that the connection $\nabla$ is compatible with the Hermitian structure on $L$.  

    Suppose, conversely, that there exists
    a Hermitian line bundle $L \to M$ with a connection $\nabla$ of curvature $\hbar^{-1}\eta$. Let $\mathcal{U}$ be a contractible open cover of $M$. For every non-empty intersection $U \cap V \neq \emptyset$ of elements $U,V \in \mathcal{U}$, denote by $c_{UV}$ the corresponding transition function of $L$. For every non-empty intersection $U\cap V \cap W \neq \emptyset$ of elements $U,V,W \in \mathcal{U}$, consider the complex valued function on $U \cap V \cap W$ given by
    \begin{equation}
    \label{eq:z}
    z_{UVW} := (\log c_{UV} + \log c_{VW} + \log c_{WU})/2\pi i,
    \end{equation}
    where the logarithms are computed using arbitrary branches; in a moment we explain the choices of branches are not important for the argument that follows. The consistency condition for the transition functions $c_{UV}$ guarantees the functions $z_{UVW}$ are integer valued, and, in particular, because they are continuous, constant. Thus, we can consider $z$ as a \v{C}ech $2$-cochain. A direct computation shows that this cochain is actually a $2$-cocycle. Although the definition of $z$ in \eqref{eq:z} depends on the choices of logarithm branches, its cohomology class $[z] \in H^2(M;\mathbb{Z})$ does not. This is the so-called \textit{first Chern class} of $L$. For every $U \in \mathcal{U}$ denote by $\theta_U$ the real valued $1$-form that determines the connection $\nabla$ on a trivialization of $L$ on $U$ according to \eqref{eq:con}. The curvature condition on $\nabla$ guarantees $d\theta_U = \hbar^{-1}\eta$ for every $U \in \mathcal{U}$. Furthermore, for every non-empty intersection $U \cap V \neq \emptyset$ of elements $U,V \in \mathcal{U}$, the glueing condition in \eqref{eq:glue} ensures
    \[
    d \log c_{UV}/2\pi i = dc_{UV}/2\pi i c_{UV} = (\theta_U - \theta_V)/2\pi.
    \]
    By the standard proof of the isomorphism between \v{C}ech and de Rham cohomology, it follows that $z$ is a representative cocycle of the class in $H^2(X;\mathbb{R})$ determined by $(2\pi \hbar)^{-1}\eta$. As $[z] \in H^2(M;\mathbb{Z})\subseteq H^2(M;\mathbb{R})$, Weil's integrality condition holds.

    Now notice that if $L \to M$ and $F \to M$ are Hermitian line bundles endowed with connections of curvature $\hbar^{-1} \eta$ and zero, i.e., flat, respectively, then $F \otimes L$ is a Hermitian line bundle on $M$ with a connection of curvature $\hbar^{-1} \eta$. Conversely, if $L \to M$ and $L' \to M$ are Hermitian line bundles endowed with connections of curvature $\hbar^{-1} \eta$, then $L' = L \otimes F$ with $F = L^{-1} \otimes L'$ a Hermitian line bundle on $M$ with a flat connection. The desired $1$-to-$1$ correspondence then follows from Theorem \ref{theo:flat}.
\end{proof}

\begin{remark}
    In the case of surfaces, the proof of the necessity of Weil's integrality condition in Theorem \ref{theo:prequant} can be reduced to an application of the curvature relation \eqref{eq:curv} over a fundamental polygon. 
    %see Figure \ref{fig:weil}.
\end{remark}

%\begin{figure}[ht]
%\centering
%\includegraphics[scale=.5]{weil2.pdf}
%\caption{A fundamental polygon of a genus $2$ surface.}
%\label{fig:weil}
%\end{figure}

%\jayadev{do we need this figure? It just shows a genus $2$ translation surface...}

\begin{remark}\label{rem:chern}
    Let $M$ be a surface. The first Chern class $[z] \in H^2(M;\mathbb{Z})$ of a line bundle $L \to M$ is a complete invariant of its isomorphism class, i.e., via the first Chern class we get a $1$-to-$1$ correspondence
    \[
    \left\lbrace \begin{array}{c}
         \text{Isomorphism classes of} \\
         \text{line bundles on $M$}
    \end{array}
    \right\rbrace \longleftrightarrow H^2(M;\mathbb{Z}).
    \]
    In particular, the proof of Theorem \ref{theo:prequant} shows that, when Weil's integrality condition holds, all Hermitian line bundle on $M$ with a connection of curvature $\hbar^{-1} \eta$ are isomorphic as line bundles (without a Hermitian structure and a connection), because their first Chern class is always $(2\pi\hbar)^{-1}[\eta] \in H^2(M;\mathbb{Z})$. Nevertheless, the discussion that follows relies crucially on the Hermitian structure and the connection on these line bundles, so the more refined correspondence provided by Theorem \ref{theo:prequant} is needed.
\end{remark}

\section{Heisenberg translation flows}\label{sec:heisenberg:flow}

We now introduce natural lifts of translation flows to Heisenberg circle bundles. We show these systems are relatively mixing with linear, hence square-integrable, decay of correlations. As an application we show that several dynamical properties of the base translation surfaces are inherited by the lifted flows. 

\subsection{Heisenberg translation flows.} Let $(M,\omega)$ be a translation surface. Unless otherwise stated, we will denote by $X$ and $Y$ the vertical and horizontal vector fields on $M$ induced by $\omega$, respectively. The corresponding flows will be denoted by
\[
\phi^X := \left\lbrace \phi^X_t \colon M \to M \right\rbrace_{t \in \mathbb{R}} \quad \text{and} \quad \phi^Y := \left\lbrace \phi^Y_t \colon M \to M \right\rbrace_{t \in \mathbb{R}}.
\]
For a general linear vector field $W = aX+bY$ (with $(a,b) \in \R^2\setminus \{(0,0)\}$), we write \[
\phi^W := \left\lbrace \phi^W_t \colon M \to M \right\rbrace_{t \in \mathbb{R}}.
\]

\begin{definition}
    Let $(M,\omega)$ be a translation surface and $C \to M$ be a Heisenberg circle bundle. The flows $\phi^X$ and $\phi^Y$ on $M$ lift via parallel transport to flows
    \[
    \Phi^{\widehat{X}} := \left\lbrace \Phi^{\widehat{X}}_t \colon C \to C \right\rbrace_{t \in \mathbb{R}} \quad \text{and} \quad \Phi^{\widehat{Y}} := \left\lbrace \Phi^{\widehat{Y}}_t \colon C \to C \right\rbrace_{t \in \mathbb{R}}
    \]
   on $C$ we refer to as \textit{Heisenberg translation flows}. We denote the vector fields corresponding to these flows by $\smash{\widehat{X}}$ and $\smash{\widehat{Y}}$ and the unit vector field in the fiber direction by $\smash{\widehat{R}}$. Given a general linear vector field $\smash{\widehat{W}} = a\widehat{X}+b\widehat{Y}+c\widehat{R}$, the associated flow 
   \[
   \Phi^{\widehat{W}} := \left\lbrace \Phi^{\widehat{W}}_t \colon C \to C \right\rbrace_{t \in \mathbb{R}}
   \]
   projects to the flow  $\phi^W$ on $M$, where $W = aX+bY$.
\end{definition}

\subsection{Heisenberg relations.} The following result shows that the Heisenberg commutation relations, compare to \eqref{eq:heisrel}, hold in this context; implicitly, this is one of the main mechanism behind the proofs of this paper.

\begin{proposition}
    \label{prop:bracket}
    Let $(M,\omega)$ be a translation surface and $C \to M$ be a Heisenberg circle bundle. Then, the Heisenberg commutation relations hold:
    \begin{equation}\label{eq:heiscomgen}
    \left[ \widehat{X},\widehat{Y} \right] = \widehat{R}, \quad \left[ \widehat{X},\widehat{R} \right] = 0, \quad \left[ \widehat{Y},\widehat{R} \right] = 0.
    \end{equation}
\end{proposition}

\begin{proof}
    We only prove the first relation; the other two relations follow by similar arguments. Fix a point $m \in M$ and identify the fiber of $C_m$ over $m$ with $\mathbb{R}/\mathbb{Z}$. Then, denoting points on $C_m$ by $(m,\rho)$ with $\rho \in \mathbb{R}/\mathbb{Z}$ and using the curvature relation \eqref{eq:curv} we get
    \begin{align*}
    \left[\widehat{X},\widehat{Y}\right](m,\rho) &= \frac{1}{2} \frac{d^2}{dt^2}\bigg|_{t=0} (\Phi^{\wY}_{-t} \circ \Phi^{\wX}_{-t} \circ \Phi^{\wY}_{t} \circ \Phi^{\wX}_{t})(m,\rho) \\
    &= \frac{1}{2} \frac{d^2}{dt^2}\bigg|_{t=0} (m,\rho+t^2)\\
    &= \widehat{R}. \qedhere
    \end{align*}
\end{proof}

\subsection{Relative mixing.}\label{sec:relmix} Given a Heisenberg circle bundle $C$, we will denote by $\mu$ the measure on $C$ that disintegrates as $\eta$ on the base translation surface $(M, \omega)$ and as the Lebesgue probability measures on the circle fibers $C_m$. This measure is invariant under the natural $U(1)$ action on $C$, i.e., invariant under the flow $\smash{\Phi^{\wR}}$. We are interested in the decay of correlations of Heisenberg translation flows $\smash{\Phi^{\wW}}$ with respect to the measure $\mu$. Because correlations on a translation surface need not necessarily decay, indeed, the linear flows $\phi^{W}$ on $(M, \omega)$ are never mixing, we formulate this problem in a relative sense. 

\begin{definition}
    Let $(M,\omega)$ be a translation surface, $\eta:= \frac{i}{2}(\omega \wedge \overline{\omega})$ the corresponding symplectic  form, and $C \to M$ a Heisenberg circle bundle. We denote by $L^2(C) = L^2(C, \mu)$ the space of $\mu$-square-integrable functions on $C$, by $L^2_*(C) \subseteq L^2(C)$ the subspace of functions that integrate to zero along $\eta$-almost-every fiber, and we identify $L^2(M)$ with the subspace of functions of $L^2(C)$ that are constant along fibers. In particular, 
    \[
    L^2(M) \perp L^2_*(C) \quad \text{and} \quad L^2(C) = L^2(M) \oplus L^2_*(C).
    \]
\end{definition}

\paragraph*{\bf Sobolev spaces} As we are interested in quantitative results for the decay of correlations, we take care to define appropriate spaces of smooth observables. 

\begin{definition}
    Let $(M,\omega)$ be a translation surface and $C \to M$ be a Heisenberg circle bundle. For every $\alpha,\beta \in \mathbb{N}$ we denote by $H^{\alpha,\beta}_\omega(C) \subseteq L^2(C)$ the Sobolev space of functions $f \in L^2(C)$ such that
    \[
    \widehat{R}^k.\widehat{Y}^j.\widehat{X}^i.f \in L^2(C) \quad \text{for all $i+j \leq \alpha$ and $k \leq \beta$}.
    \]
\end{definition}

\paragraph*{\bf Fourier modes} As discussed above, any Heisenberg circle bundle $C \to M$ is a $U(1)$-bundle and the $U(1)$-action corresponds to the flow 
\[
\Phi^{\wR} := \{\Phi_t^{\wR} \colon C \to C\}_{t \in \mathbb{R}}
\]
generated by the vector field $\smash{\widehat{R}}$. This flow preserves the measure $\mu$ on $C$. In particular, this allows us to decompose $L^2(C)$ into Fourier modes:

\begin{definition}
    Let $(M,\omega)$ be a translation surface and $C \to M$ be a Heisenberg circle bundle. For every $n \in \mathbb{Z}$ define
    \[
    E_n := \{f \in L^2(C) \ | \ f \circ \Phi^{\wR}_t = e^{2\pi i n t} f\}.
    \]
    Then, we have an on orthogonal splitting
    \begin{equation}
    \label{eq:fourieraa}
    L^2(C) = \bigoplus_{n \in \mathbb{Z}} E_n.
    \end{equation}
    The projections $\pi_n \colon L^2(C) \to E_n$ of this splitting are given by
    \[
    f \in L^2(C) \mapsto \pi_n(f) := \int_{0}^1 e^{-2\pi i n t} (f \circ \Phi^{\wR}_t) \thinspace dt \in E_n.
    \]
    In particular,
    \[
    E_0 = L^2(M) \quad \text{and} \quad \pi_0(f) = \int_{0}^1 f \circ \Phi^{\wR}_t \thinspace dt.
    \]
\end{definition}

\begin{remark}   Proposition \ref{prop:bracket} guarantees the decomposition in \eqref{eq:fourieraa} is invariant with respect to $\Phi^{\wW}$ for any linear vector field $\wW = a\wX + b\wY+ c\wR$.
\end{remark}

\paragraph*{\bf Main result.} We are now ready to state and prove the main result of this section: 

\begin{theorem}
    \label{thm:rel_mix_new}
    Let $(M,\omega)$ be a translation surface and $C \to M$ be a Heisenberg circle bundle. Suppose
    \[
    \widehat{W} := a \widehat{X} + b \widehat{Y} + c \widehat{R}, \quad \text{with } a,b,c \in \mathbb{R}.
    \]
    Then, if $(a,b)\not =(0,0)$, the flow $\Phi^{\wW}$ is relatively mixing with respect to the projection $C \to M$, i.e., for every $f \in L^2_*(C)$ and every $g \in L^2(C)$,
    \[
    \lim_{t \to \pm \infty} \left\langle f, g \circ \Phi^{\wW}_t \right\rangle_{L^2(C)} = 0.
    \]
     Furthermore, for every $\alpha \in \N\setminus \{0\}$ there exist a constant $K_\alpha>0$ such that for every $f \in L^2_*(C) \cap H^{\alpha,0}_{\omega}(C)$, every $g\in H^{\alpha,0}_{\omega}(C)$, and every $t\in \R$,
     $$
     \bigg\vert \left\langle f, g  \circ \Phi^{\wW}_t\right\rangle_{L^2(C)} \bigg\vert \leq \frac{K_\alpha}{(1+\vert t\vert)^\alpha} \Vert f \Vert_{H^{\alpha,0}_{\omega}(C)} \Vert g \Vert_{H^{\alpha,0}_{\omega}(C)}.
     $$
     In particular, the correlation functions are square integrable on $\R$.
\end{theorem}

\begin{remark}
    In the case of Heisenberg nilflows, i.e., when the base surface is a torus, Theorem \ref{thm:rel_mix_new} is well known and can be proved using the representation theory of the Heisenberg group. We highlight that such strong algebraic tools are not available in the case of general Heisenberg translation flows, so we need to appeal to more general arguments.
\end{remark}

\begin{proof}
Let us first give the intuition behind the proof, following the shearing arguments of Marcus \cite{Marcus}. First, one decomposes the correlation function in question as an integral along segments in a given direction transverse to the flow. Then, as one pushes these segments using the flow, they get sheared along the fiber direction more and more; this phenomena relies crucially on Proposition \ref{prop:bracket} and, thus, on the underlying curvature assumption in the Definition \ref{def:hcb}. Under the hypothesis that the integrals along fibers vanish, the desired correlation function decays to zero, see Figure \ref{fig:shear}.

\begin{figure}[ht]
\centering
\begin{tikzpicture}[scale = 1.5]
  \draw[thick,->](0,0)--(1,0);  
  \draw[thick,->](1,0)--(3,0);  
    \draw[thick](3,0)--(4,0)node[right]{\tiny $X$}; 
    \draw[thick, red](0,0)--(1, 1);
    \draw[thick](1,1)--(1, 2)--(0, 1)node[left]{\tiny $R$}--(0,0);
       \draw[thick](2,0)--(3, 1);
       \draw[thick, red](2,0)--(3,2);
    \draw[thick](3,1)--(3, 2)--(2, 1)--(2,0);
       \draw[thick](4,0)--(5, 1)node[right]{\tiny $Y$};
    \draw[thick](5,1)--(5, 2)--(4, 1)--(4,0);
       \draw[thick, red](4,0)--(4.5,1.5);
              \draw[thick, red](4.5,.5)--(5,2);

\end{tikzpicture}
\caption{The shearing argument.}
\label{fig:shear}
\end{figure}
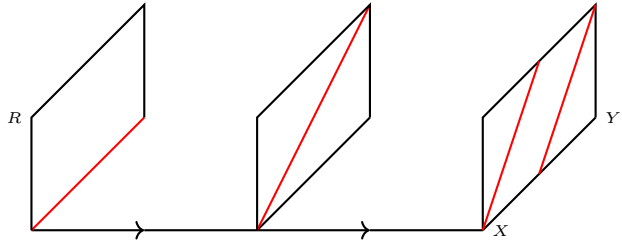
%\jayadev{I will redraw this in TikZ, and relabel Z as R}

To make this argument precise, consider $f \in L^2(C)$ and $g \in L^2_*(C)$. Without loss of generality we assume that $a \neq 0$ and, after rescaling the vector field, we can further assume that $a =1$, so we have
\begin{equation}
\label{eq:bracket}
\left[\widehat{W}, \widehat{Y}\right] = \left[\widehat{X}, \widehat{Y}\right] = \widehat{R}.
\end{equation}

Since the measure $\mu$ is invariant under the flow $\Phi^Y$, for every $\delta >0$ we have
\begin{align*}
\left\langle f, g \circ \Phi^{\wW}_t \right\rangle_{L^2(C)} &= \delta^{-1} 
\int_0^\delta \langle f\circ \Phi^Y_s , g \circ \Phi^{\wW}_t \circ \Phi^Y_s\rangle \thinspace ds\\ 
&=\delta^{-1} \int_0^\delta \int_{C} (f\circ \Phi^Y_s) \overline{(g \circ \Phi^{\wW}_t \circ \Phi^Y_s)} \thinspace d\mu \thinspace ds \\
&= \delta^{-1}  \int_{C} \int_0^\delta (f\circ \Phi^Y_s) \overline{(g \circ \Phi^{\wW}_t \circ \Phi^Y_s)} \thinspace ds \thinspace d\mu.
\end{align*}

Since the subspace $H^{1,1}_\omega(C) \subseteq L^2(C)$ is dense (this follows from the corresponding density of $H^{1,1}_{M}$ in $L^2(M)$, the Heisenberg relations \eqref{eq:heiscomgen}, and the fact that at the singular points we have the structure of a finite cover; see ~\cite{Beals} for a general discussion), we can assume without loss of generality that $\widehat{Y}.f \in L^2(C)$. Applying integration by parts to the inner integral with
\begin{align*}
    u &= f\circ \Phi^Y_s, \\
    dv &= \overline{(g \circ \Phi^{\wW}_t \circ \Phi^{\wY}_s)} \thinspace ds,\\
    du &= (\widehat{Y}.f) \circ \Phi^{\wY}_s \thinspace ds,\\
    v &= \int_{0}^s \overline{(g \circ \Phi^{\wW}_t \circ \Phi^{\wY}_\sigma)} \thinspace d\sigma,
\end{align*}
we get
\begin{align}
    \left\langle f, g \circ \Phi^{\wW}_t \right\rangle_{L^2(C)} =  \ &\delta^{-1} \left\langle f\circ \Phi^{\wY}_s,\int_{0}^\delta (g \circ \Phi^{\wW}_t \circ \Phi^{\wY}_\sigma) \thinspace d\sigma \right\rangle_{L^2(C)} \label{eq:corr}\\
    &-\delta^{-1} \int_0^\delta  \left\langle (\widehat{Y}.f) \circ \Phi^{\wY}_s, \int_{0}^s (g \circ \Phi^{\wW}_t \circ \Phi^{\wY}_\sigma) \thinspace d\sigma\right\rangle_{L^2(C)} \thinspace ds. \nonumber
\end{align}
Therefore, it is enough to prove that, for every $\delta > 0$,
\begin{equation}
\label{eq:conv_curves}
\lim_{t \to \pm \infty} \int_{0}^\delta (g \circ \Phi^{\wW}_t \circ \Phi^{\wY}_s) \thinspace ds = 0 \quad \text{in } L^2(C).
\end{equation}

Again, we can assume $f \in H^{1,1}_\omega(C)$. Then, by \eqref{eq:bracket}, for all $s,t \in \mathbb{R}$,
\begin{equation}
\label{eq:ODE}
\frac{d}{ds} (g \circ \Phi^{\wW}_t \circ \Phi^{\wY}_s) = \left(\left(\widehat{Y}+t\widehat{R}\right).g\right) \circ \Phi^{\wW}_t \circ \Phi^{\wY}_s.
\end{equation}
Indeed, 
\[
\frac{d}{dt}d\Phi^{\wW}_t \left(\widehat{Y}\right) = d\Phi^{\wW}_t\left(\left[\widehat{W},\widehat{Y}\right]\right) =  d\Phi^{\wW}_t\left(\widehat{R}\right) = \widehat{R},
\]
from where, after integration over $t$, we deduce
\[
d\Phi^{\wW}_t\left( \widehat{Y} \right) = \widehat{Y} + t \widehat{R}.
\]

For every $n \in \mathbb{N}$, if $g \in E_n$, then equation $\eqref{eq:ODE}$ has the form
\[
\frac{du}{ds} = a + 2\pi i n t u,
\]
where
\begin{align*}
    u &= g \circ \Phi_t^{\wW} \circ \Phi_s^{\wY}, \\
    a &= (\widehat{Y}.g) \circ\ \Phi_t^{\wW} \circ \Phi_s^{\wY}.
\end{align*}
The solution of this ordinary differential equation is given by
\[
u = e^{2 \pi i  n t s} \left( u(0) + \int_0^s e^{-2 \pi i n t \sigma} a(\sigma) \thinspace d \sigma \right).
\]
It follows that
\[
g \circ \Phi_t^{\wW} \circ \Phi_s^{\wY} = e^{2 \pi i  n t s} \left( g \circ \Phi_t^{\wW} + \int_0^s e^{-2 \pi i n t \sigma} \thinspace (\widehat{Y}.g) \circ\ \Phi_t^{\wW} \circ \Phi_\sigma^{\wY} \thinspace d \sigma \right).
\]
For $n \neq 0$, integrating $s$ between $0$ and $\delta$ and using integration by parts with
\begin{align*}
u &= \int_0^s e^{-2 \pi i n t \sigma} \thinspace(\widehat{Y}.g) \circ \Phi_t^{\wW} \circ \Phi_\sigma^Y \thinspace d \sigma, \\
dv &= e^{2 \pi i nt s},\\
du &= e^{- 2 \pi i n t s}  (\widehat{Y}.g) \circ\ \Phi_t^{\wW} \circ \Phi_s^{\wY},\\
v &= \frac{1}{2 \pi i n t} e^{2 \pi i n t s},
\end{align*}
we deduce that, 
\begin{align}
    \int_0^\delta g \circ \Phi_t^{\wW} \circ \Phi_s^{\wY} \thinspace ds = &\frac{1}{2 \pi i n t} (g \circ \Phi_t^{\wW})(e^{2 \pi i n t \delta}  
 - 1) \label{eq:int_id} \\
 &+ \frac{1}{2 \pi i n t} \left( \int_0^\delta e^{2 \pi i n t (\delta - \sigma)} \thinspace \thinspace(\widehat{Y}.g) \circ \Phi_t^{\wW} \circ \Phi_\sigma^{\wY} \thinspace d \sigma  \right.  \nonumber\\
 & \hspace{3cm}- \left.\int_0^\delta (\widehat{Y}.g) \circ\ \Phi_t^{\wW} \circ \Phi_s^{\wY} \thinspace ds \right). \nonumber
 \end{align}
This integral converges to $0$ as $t \to \pm \infty$, proving the desired decay of correlations.

The stated bound on correlations for $\alpha=1$ follows from equations \eqref{eq:corr} and \eqref{eq:int_id}.  For $\alpha >1$ one can proceed by induction
since, by the same equations, one can 
express the desired correlation functions
in terms of integrals with respect to $s$ over bounded intervals of linear combination of  expressions of the form 
\[
\frac{1}{2 \pi i n t} \left\langle (\widehat{Y}.f) \circ \Phi_s^Y
, (\widehat{Y}.g) \circ \Phi^{\wW}_t \right\rangle_{L^2(C)} \,.
\]
Thus, an induction argument proves the dersired polynomial decay of
correlations for all powers $\alpha \in \N$ under the hypothesis
that $\smash{\widehat{Y}^\alpha.f}, \thinspace\smash{\widehat{Y}^\alpha.g} \in L^2(C)$.
\end{proof}

\subsection{Applications.}\label{sec:applications} Let $(M,\omega)$ be a translation surface and $C \to M$ be a Heisenberg circle bundle. Recall that, given a vector field on $C$ of the form
    \[
    \widehat{W} := a \widehat{X} + b \widehat{Y} + c \widehat{R}, \quad \text{with }a,b,c \in \mathbb{R},
    \]
    we denote by
    \[
    \phi^W := \{\phi^W_t \colon M \to M \}_{t \in \mathbb{R}}
    \]
    the flow induced by the vector field $W = aX + bY$ on $M$.

\begin{corollary}
\label{cor:ethry}
Let $(M,\omega)$ be a translation surface, $\eta:= \frac{i}{2}(\omega \wedge \overline{\omega})$ be its corresponding area form, $C \to M$ be a Heisenberg circle bundle, and $\mu$ be the measure on $C$ as defined above. Consider a vector field on $C$ of the form
    \[
    \widehat{W} := a \widehat{X} + b \widehat{Y} + c \widehat{R}, \quad \text{with } a,b,c \in \mathbb{R},
    \]
and let $W = a\wX + b\wY$ be the projected vector field on $(M,\omega)$. If the flow $\phi^W$ on $M$ is any of the following,
\begin{itemize}
    \item minimal,
    \item ergodic with respect to $\eta$,
    \item uniquely ergodic,
    \item weakly mixing with respect to $\eta$,
\end{itemize}
then the flow $\Phi^{\wW}$ on $C$ is, respectively,
\begin{itemize}
    \item minimal,
    \item ergodic with respect to $\mu$,
    \item uniquely ergodic,
    \item weakly mixing with respect to $\mu$.
\end{itemize}
\end{corollary}

\begin{proof}  
Let us assume $\phi^W$ is minimal. By Theorem \ref{thm:rel_mix_new}, there is a one-to-one correspondence between the ergodic decomposition of $\eta$ with respect to $\phi^W$ and the ergodic decomposition of $\mu$ with respect to $\smash{\Phi^{\wW}}$ obtained by considering Lebesgue as the conditional measure along the fibers. Indeed, denote by $\eta'$ any ergodic component of $\eta$ with respect to $\phi^W$. Since $\phi^W$ is a  minimal translation flow, the ergodic decomposition of $\eta$ is finite. In particular,
there exists a function $\chi'$ on $M$, the characteristic function of a $\phi^W$-invariant set, such that $\eta' = \chi' \eta$.
Let $\mu':= \chi' \mu$ denote the measure on $C$ which projects to $\eta'$ on $M$ and has the normalized Lebesgue measure as the conditional measure along fibers. We claim that $\mu'$ is an invariant and ergodic measure for $\smash{\Phi^{\wW}}$.

Given a function $f \in L^2(C)$ we consider the splitting
\begin{equation*}
f = f_0 + f^\perp_0 \quad  \text{ with } \quad f_0 \in E_0 \quad \text{ and } \quad 
f^\perp_0 \in E_0^\perp= \bigoplus_{n\in  \Z\setminus\{0\}} E_n\,.
\end{equation*}
Since $f_0$ is constant along fibers, we can consider it as a function on $M$, and, using the ergodicity of $\phi^W$ with respect to $\eta'$, deduce that, in $L^2(C)$,
\begin{align*}
\lim_{T \to \infty} \frac{1}{T} \int_0^T \chi' f_0 \circ \Phi^{\wW}_t \thinspace dt &=
\lim_{T \to \infty} \frac{1}{T}\int_0^T \chi'f_0 \circ \phi^W_t \thinspace dt 
= \int_M \chi' f_0 \thinspace d\eta
= \int_C \chi' f \thinspace  d\mu'.
\end{align*}
Next, by Theorem \ref{thm:rel_mix_new}, for any function $g\in L^2(C)$,
$$
\lim_{T \to \infty} \left\langle \frac{1}{T} \int_0^T \chi' f^\perp_0 \circ \Phi^{\wW}_t \thinspace dt, g \right\rangle_{L^2(C)} = \lim_{T \to \infty}\frac{1}{T} \int_0^T \left\langle \chi' f^\perp_0 \circ \Phi^{\wW}_t, g \right\rangle_{L^2(C)} \thinspace  dt = 0.
$$
In other words, the ergodic averages with respect to $\Phi^{\wW}$ of the function $\chi' f^\perp_0$
converge weakly to zero. Since by the mean ergodic theorem the corresponding
limit exists in $L^2(C)$, such a limit must be equal to 
zero. We have thus proved that, in $L^2(C)$,
$$
\lim_{T \to \infty} \frac{1}{T} \int_0^T \chi' f \circ \Phi^{\wW}_t \thinspace dt = \int_C \chi' f \thinspace  d\mu',
$$
i.e., the flow $\Phi^{\wW}$ is ergodic with respect to $\mu'$.

Let us now show that $\smash{\Phi^{\wW}}$ is transitive. Denote by $\eta'$ any ergodic component with respect to $\phi^W$ of $\eta$ and by $\mu'$ the corresponding ergodic component of $\mu$ with respect to $\smash{\Phi^{\wW}}$. Because $\phi^W$ is minimal, $\eta'$ has full support. Thus, $\mu'$ also has full support. Ergodicity of $\mu'$ then guarantees the existence of a generic point $(m_0, \rho_0) \in C$ for $\mu'$ whose $\smash{\Phi^{\wW}}$-orbit must be dense in $C$. Now consider the fiber $C_{m_0}$ of $C$ containing $(m_0, \rho_0)$. As rotation along fibers commutes with $\smash{\Phi^{\wW}}$, every point of the fiber has dense $\smash{\Phi^{\wW}}$-orbit. To deduce that any point $(m,\rho) \in C$ has dense $\smash{\Phi^{\wW}}$-orbit, and hence $\smash{\Phi^{\wW}}$ would be minimal as desired, we use the minimality of $\phi^W$ to find points along the $\smash{\Phi^{\wW}}$-orbit of $(m,\rho)$ that get arbitrarily close to the fiber $C_{m_0}$. A compactness argument then allows us to find a subsequence converging to a point $(m_0,\rho_1) \in C_{m_0}$. As explained above, such point has dense $\smash{\Phi^{\wW}}$-orbit. The desired conclusion then follows by the triangle inequality. 

Next, assume $\Phi^W$ is ergodic with respect to $\eta$. By the correspondence established above between ergodic components of 
$\eta$ with respect to $\Phi^W$ and the ergodic components of $\mu$ with respect to $\smash{\Phi^{\wW}}$, it follows that
$\smash{\Phi^{\wW}}$ is ergodic with respect to $\mu$.

Let us now assume that $\phi^W$ is uniquely ergodic. Notice that Proposition \ref{prop:bracket} guarantees the flows $\smash{\Phi^{\wW}}$ and $\smash{\Phi^{\wR}}$ commute. Now, by the above argument, the flow $\smash{\Phi^{\wW}}$ is ergodic with respect to $\mu$. Thus, by Furstenberg's criterion on unique ergodicity of skew products \cite[Theorem 4.1]{Fu61}, we conclude $\smash{\Phi^{\wW}}$ is uniquely ergodic.

%Since
%by the commutation relations the fiber rotation flow $\phi^{\hat R}_\R$ commutes with $\phi^{\hat W_{a,b,c}}_\R$, which is ergodic
%by the above argument, the flow $\phi^{\hat W_{a,b,c}}_\R$ is uniquely ergodic by Furstenberg's criterion on the unique ergodicity
%of skew-products \cite[Theorem 4.1]{Fu61}.

%\smallskip
Next, let us assume that $\phi^W$ is weakly mixing with respect to $\eta$. Consider arbitrary functions $f,g \in L^2(C)$ of zero average. For every $T > 0$ we write
\begin{gather*}
\frac{1}{T} \int_0^T \bigg\vert \left\langle f\circ \Phi^{\wW}_t, g \right\rangle_{L^2(C)} \bigg\vert \thinspace dt \\ 
\leq \frac{1}{T} \int_0^T \bigg\vert \left\langle f_0\circ \Phi^{\wW}_t, g \right\rangle_{L^2(C)} \bigg\vert \thinspace dt + \frac{1}{T} \int_0^T \bigg\vert \left\langle f_0^\perp\circ \Phi^{\wW}_t, g \right\rangle_{L^2(C)} \bigg\vert \thinspace dt.
\end{gather*}
Since, by assumption, the flow $\phi^W$ is weakly mixing with respect to $\eta$,
considering $f_0$ and $g_0$ as zero average functions on $L^2(M)$, we deduce
$$
\lim_{T \to \infty} \frac{1}{T}\int_0^T \bigg\vert \left\langle f_0\circ \Phi^{\wW}_t, g \right\rangle_{L^2(C)} \bigg\vert \thinspace dt= \lim_{T \to \infty} \frac{1}{T}\int_0^T \big\vert \left\langle f_0\circ \phi^W_t, g_0 \right\rangle_{L^2(M)} \big\vert \thinspace dt = 0.
$$
Theorem \ref{thm:rel_mix_new} guarantees
$$
 \lim_{T \to \infty} \frac{1}{T} \int_0^T \bigg\vert \left\langle f_0^\perp\circ \Phi^{\wW}_t, g \right\rangle_{L^2(C)} \bigg\vert \thinspace dt = 0 .
$$
We conclude that 
$$
\lim_{T \to \infty} \frac{1}{T} \int_0^T \bigg\vert \left\langle f\circ \Phi^{\wW}_t, g \right\rangle_{L^2(C)} \bigg\vert \thinspace dt =0,
$$
i.e., the flow $\Phi^{\wW}$ is weakly mixing with respect to $\mu$.
\end{proof}

\begin{question} Is the speed of converge in the ergodic theorem
for the flow $\Phi^{\widehat W}$ polynomial, i.e., a power law, for
Lebesgue generic parameters $(a,b,c) \in \smash{\mathbb{P}^1(\R^3)}$ and for sufficiently smooth functions, as is the case for translation flows on higher genus surfaces?

\end{question}

\section{Spectral theory}\label{sec:spectral}

We now turn to study the spectra of the Koopman groups of Heisenberg translation flows. We prove that if the linear flow on the base translation surface is aperiodic, then the maximal spectral type of these groups is positive on all open sets. 

\subsection{Spectra of Koopman groups.} Let us recall some terminology about general Koopman groups and their spectral theory. We will consider probability measure-preserving flows $(X,\mu,\phi)$, where $X$ is a measurable space, $\mu$ is a probability measure, and $\phi$ is a measurable flow on $X$ preserving $\mu$; we write
\[
\phi := \{\phi_t \colon X \to X\}_{t \in \mathbb{R}}.
\]

\begin{definition}
    Let $(X,\mu,\phi)$ be a probability measure-preserving flow. Then, the \textit{Koopman group} of $\phi$ is the one-parameter strongly continuous group 
    \[
    U := \{U_t \colon \ L^2(X,\mu) \to L^2(X,\mu)\}_{t \in \mathbb{R}}
    \]
    of unitary operators  given for every $t \in \mathbb{R}$ and every $f \in L^2(X,\mu)$ by
    \[
    U_t(f) := f \circ \phi_t.
    \]
    By Stone's theorem, there exists a unique unbounded self adjoint operator 
    \[
    H \colon L^2(X,\mu) \to L^2(X,\mu)
    \]
    that generates the Koopman group, i.e., such that for every $t \in \mathbb{R}$,
    \[
    U_t = e^{-itH}.
    \]
    We refer to $H$ as the \textit{Koopman Hamiltonian} of $\phi$ and to its spectrum as the \textit{spectrum} of $\phi$. In the same way, spectral properties of $H$ will be referred to as spectral properties of $\phi$.
\end{definition}

\subsection{Rokhlin towers.} To prove the main result of this section we will use Rokhlin's lemma for aperiodic probability measure-preserving flows, which we now recall:

\begin{definition}
    A measure preserving flow $(X,\mu,\phi)$ is said to be \textit{aperiodic} if the set of its periodic points has measure zero.
\end{definition}

\begin{definition}
    Let $(X,\mu,\phi)$ be a measure preserving flow. A \textit{Rokhlin tower} is a family of disjoint subsets of $X$ of the form
    \[
    \{\phi_t S \colon \ t \in [0,T)\},
    \]
    where $S \subseteq X$ and $T > 0$.
    The set $S$ is called the \textit{base} of the tower, $T$ is called the \textit{height} of the tower, the set 
    \[
    R := \bigcup_{t \in [0,T)} \phi_t S   
    \]
    is called the \textit{tower}, and $X \setminus R$ is called the \textit{error set}.
\end{definition}

\begin{lemma}
    \label{lem:rokhlin}
    Suppose $(X,\mu,\phi)$ is a probability measure-preserving flow on a Lebesgue probability space. Then, $\phi$ admits  Rokhlin towers of arbitrarily large heights with error sets of arbitrarily small measure.
\end{lemma}

In this context, the following general result holds:

\begin{lemma} 
\label{lemma:no_gaps}
Suppose $(X,\mu,\phi)$ is a probability measure-preserving flow on a Lebesgue probability space. If $\phi$ is aperiodic, then its spectrum is the entire real axis, and, in particular, its maximal spectral type has full support.
\end{lemma}

\begin{proof} 
%By Lemma \ref{lem:rokhlin}, $\phi$ admits towers of arbitrarily large heights with error sets of arbitrarily small measure.  Let $\lambda$ be any
%real number. We claim that $\lambda$ belongs
%to the spectrum of $(M, \phi_\R, \mu)$.
Let $H$ be the Koopman Hamiltonian of $\phi$. Fix $\lambda \in \mathbb{R}$. To prove that $\lambda$ belongs to the spectrum of $H$, it is enough to find a sequence $(f_k)_{k \in \mathbb{N}}$ of functions in $L^2(X,\mu)$ with $\|f_k\|_{L^2(X,\mu)} = 1$ and such that for every $t > 0$,
\begin{equation}
\label{eq:spec_lim}
\lim_{k \to \infty} \|f_k \circ \phi_t - e^{-i\lambda t} f_k\|_{L^2(X,\mu)} = 0.
\end{equation}

Let $k \in \mathbb{N}$. By Lemma \ref{lem:rokhlin}, $\phi$ admits a Rokhlin tower $R_k$ of height at least $k$ and with error set of measure at most $1/k$. Let $B_k$ be the base of such tower and $T_k \geq k$ be the height. For every $x \in B_k$ and every $t \in [0,T_k)$ set
\[
f_k(\phi_t(x)) := e^{-i \lambda t} \mu(R_k)^{-1/2}.
\]
On the other hand, if $x \in X \setminus R_k$ set
\[
f_k(x) := 0.
\]
Then, $\|f_k\|_{L^2(X,\mu)} = 1$. Furthermore, for all $t > 0$,
\[
\|f_k \circ \phi_t - e^{-i\lambda t} f_k\|_{L^2(X,\mu)} \leq  2(t/T_k)^{1/2}.
\]
Taking $k \to \infty$ we conclude that \eqref{eq:spec_lim} holds.
%Let $T>0$ and let $B_T\subset M$ be the basis of a tower of height $T$. For all $S<T$, let $R_S \subset M$ denote the subtowers, that is, the images of the injective maps
%$$
%(x,t) \to \phi_t(x) \quad \text{ for } (x,t) \in B\times [0,S]\,.
%$$
%Let $F^{(T)}_\lambda$ be the function defined
%as
%$$
%F^{(T)}_\lambda (\phi_t(x) ) = \begin{cases} e^{i \lambda t} \mu (R_{T/2})^{-1/2} \quad &\text{for } (x,t) \in B \times [0,T/2]\,, \\
%0  \quad &\text{on } M\setminus R_{T/2}.
%\end{cases}
%$$
%By its definition the function $F^{(T)}_\lambda$ has unit norm
%in $L^2(M, \mu)$ and, for all $s\in [0,T/2]$,
%$$
%\Vert F^{(T)}_\lambda \circ \phi_s - e^{i\lambda s} F^{(T)}_\lambda \Vert_{L^2(M, \mu)} \leq  2 \Big(\frac{s}{T}\Big)^{1/2}\,.
%$$
%Thus for every $\lambda \in \R$ there exists a sequence 
%$\{F^{(k)}_\lambda\}_{k\in %\N}$ of normalized square-integrable functions such 
%that, for every $s> 0$, 
%$$
%\Vert F^{(k)}_\lambda \circ \phi_s - e^{i\lambda s} F^{(k)}_\lambda \Vert_{L^2(M, \mu)} \to 0\,,
%$$
%which implies that $\lambda$ belongs to the spectrum of the flow.
\end{proof}

\subsection{Spectra of translation flows} Now let $(M,\omega)$ be a translation surface. Recall that we denote by $X$ and $Y$ the vertical and horizontal vector fields induced by $\omega$ on $M$. For every $a,b \in \mathbb{R}$ consider the vector field
$W= W_{a,b} = a X+ bY$ on $M$. It is well known that for all but countably many $(a,b)\in P^1(\R)$, the linear flow $\phi^{W_{a,b}}$ on $M$ is aperiodic. It follows from Lemma \ref{lemma:no_gaps}
that the spectrum of such flows is all of $\mathbb{R}$. Furthermore, in \cite{AHCF24}, a complete characterization of the translation surfaces for which the spectrum
of $\phi^{W_{a,b}}$ is not continuous for Lebesgue almost every $(a,b)\in P^1(\R)$ is provided.

\subsection{Spectra of Heisenberg flows.} Let $(M,\omega)$ be a translation surface and $C \to M$ be a Heisenberg circle bundle. Recall that for every $n \in \mathbb{Z}$ we denote
    \[
    E_n := \{f \in L^2(C) \ | \ f \circ \Phi^{\wR}_t = e^{2\pi i n t} f\}
    \]
    and that we consider the Fourier orthogonal splitting
    \begin{equation*}
    %\label{eq:fourieraa}
    L^2(C) = \bigoplus_{n \in \mathbb{Z}} E_n.
    \end{equation*}
    Recall also that we denote
    \[
    L^2_*(C) := E_0^\perp = \bigoplus_{n \in \mathbb{Z} \setminus \{0\}} E_n .
    \]

\begin{theorem} 
\label{prop:spectra}
Let $(M,\omega)$ be a translation surface and $C \to M$ be a Heisenberg circle bundle. Consider the  vector field on $C$ given by 
    \[
    \widehat{W} := a \widehat{X} + b \widehat{Y} + c \widehat{R}, \quad \text{with } a,b,c \in \mathbb{R},
    \]
and let $W = aX + bY$ be the projected vector field on $M$. Then, for every $(a,b,c)\in P^1(\R^2)$ such that $(a,b)\not=(0,0)$, the flow $\smash{\Phi^{\wW}}$ has absolutely continuous spectrum of countable multiplicity, and, if the projected translation flow $\phi^{W}$ is aperiodic, then the spectrum on the space $L^2_*(C)$ has no gaps, i.e., the maximal spectral type of $\smash{\Phi^{\wW}}$ on this space has full support. In fact, for every $n\in \Z\setminus 
\{0\}$, if $(a,b)\not=(0,0)$, then the spectrum of $\smash{\Phi^{\wW}}$ is absolutely continuous on $E_n$ and, if the projected translation flow $\phi^{W}$ is aperiodic, then such spectrum has no gaps.
\end{theorem}
\begin{proof}
For every $n\in \Z$ the space $E_n$ is invariant under all flows
$\Phi^{\wW}$ and, by Theorem~\ref{thm:rel_mix_new}, for $n\not=0$, correlations of smooth functions in $E_n$ are square-integrable. It follows that $\smash{\Phi^{\wW}}$ has absolutely continuous spectrum on $E_n$ and, in particular, has absolutely continuous
spectrum of countable multiplicity on $L^2_*(C)$. Indeed, for every $f, g \in L^2(C)$, the correlation function 
\[
t \in \mathbb{R} \mapsto \langle f \circ \Phi^{\wW}_t, g\rangle_{L^2(C)}
\]
is the Fourier
transform of the spectral measure induced by $f$ and $g$, seen as a complex measure on the real line. Thus, if the correlation function
is square-integrable, then its Fourier transform is not only a measure but actually a square-integrable function. Therefore,
the corresponding spectral measure is absolutely continuous with  square-integrable density. Since for every $n \in \Z$
smooth functions in $E_n$ are dense in $E_n$ and the spectral components of the flow are
closed, it follows that the spectrum of $\smash{\Phi^{\wW}_\R}$ on $E_n$ for $n\not =0$ is absolutely continuous with respect to Lebesgue.

The proof that the spectrum on each $E_n$ has no gaps for $n\not =0$ is similar to the proof of Lemma \ref{lemma:no_gaps}.
Indeed, fix $\lambda \in \mathbb{R}$. Since by hypothesis the projected flow $\phi^W$ is aperiodic, there exists a sequence $(f_k)_{k\in \N}$ of functions in $L^2(M)$ with $\|f_k\|_{L^2(M)} =1$ and such that for every $s>0$ we have
\[
\lim_{k \to \infty} \|f_k \circ \phi_s^W - e^{-i\lambda s} f_k\|_{L^2(X,\mu)} = 0.
\]
In this setting we can further assume that the support $R_k$ of the functions $f_k$ is simply connected. It follows that we can trivialize the circle bundle $C$ over $R_k$ and consider a section $\theta_k \colon R_k \to C$ equivariant with respect to $\phi^W$ and $\smash{\Phi^{\wW}_\R}$.
Consider then the sequence of functions $\smash{(\widehat f_k)}_{k \in \mathbb{N}}$ defined as
\[
\widehat f_k(\Phi^{\widehat{R}}_t( \theta_k(x))) := e^{int}f_k(x) \quad \text{if } x \in R_k \text{ and }t \in \mathbb{R},
\]
and zero otherwise. Then, $f_k \in E_n$, $\|f_k \|_{L^2(C)} =1$, and for every $s \in \mathbb{R}$,
\[
\lim_{k \to \infty} \|\widehat{f}_k \circ \phi_s^{\widehat{W}} - e^{-i\lambda s} \widehat{f}_k\|_{L^2(C)} = 0. \qedhere
\]
\end{proof}

\begin{question} Is the spectrum of $\Phi^{\widehat W}$ on $L^2_*(C)$ Lebesgue
with infinite (countable) multiplicity ? In other terms, is the spectrum of $\Phi^{\widehat W}$ on $E_n$ Lebesgue for all $n \in \Z\setminus\{0\}$~? We remark that the criterion for countable Lebesgue spectrum of \cite{FFK21} does not immediately apply to our case.
    
\end{question}

\section{Prequantization}\label{sec:prequant}

We now turn to the question of how Heisenberg translation flows can be interpreted as prequantizations of translation flows on surfaces. In particular, the corresponding covariant derivative operator can be interpreted as the prequantized version of the momentum in the direction of the translation flow.

\subsection{Dirac's quantum conditions.}\label{sec:dirac} Let $(M,\omega)$ be a translation surface, $\eta:= \frac{i}{2}(\omega \wedge \overline{\omega})$ be its corresponding area form. Even though $(M,\omega)$ is singular as a symplectic manifold, one can still make sense of the Poisson bracket on it in the following way:

\begin{definition}
    \label{def:poisson}
    Let $(M,\omega)$ be a translation surface with singularities $\Sigma \subseteq M$ and $\eta:= \frac{i}{2}(\omega \wedge \overline{\omega})$ be its corresponding area form. Given a smooth compactly supported function $f \in C_0^\infty(M \setminus \Sigma)$, its \textit{Hamiltonian vector field} is the vector field $X_f$ defined on $M \setminus \Sigma$ by
    \[
    \eta(\cdot,X_f) = df.
    \]
    Given a pair of smooth compactly supported functions $f,g \in C_0^\infty(M \setminus \Sigma)$, their \textit{Poisson bracket} is the smooth function $\{f,g\} \in C_0^\infty(M \setminus \Sigma)$ defined by the relation
    \[
    \{f,g\} := \eta(X_f,X_g).
    \]
\end{definition}

\paragraph*{\bf Square-integrable sections} Let $(M,\omega)$ be a translation surface, $\eta:= \frac{i}{2}(\omega \wedge \overline{\omega})$ be its corresponding area form, and $L \to M$ be a Heisenberg line bundle. Denote by $\Gamma^2(L)$ the space of square-integrable sections of $L$, i.e., sections $s \colon M \to L$ such that
\[
\|s\|_{\Gamma^2(L)} :=\int_M \|s(m)\|^2 \thinspace d\eta(m) < +\infty.
\]
Such sections can be interpreted as prequantum states. More precisely, every such section gives rise to a probability distribution on classical states, i.e., on M, given by
\[
\frac{\|s(m)\|^2}{\|s\|_{\Gamma^2(L)}} \thinspace d\eta(m).
\]

\paragraph*{\bf Observables} Classical observables of $(M,\omega)$ correspond to square-integrable functions on $M$. In what follows $L^2(M)$ will denote the space of \textit{real} square-integrable functions on $M$; it is important that the functions considered are real and not complex. Prequantum observables correspond to unbounded self-adjoint operators on the Hilbert space $\Gamma^2(L)$; denote the group of such operators by $\mathfrak{u}(\Gamma^2(L))$. Prequantization seeks an assigment between classic and quantum observables satisfying \textit{Dirac's quantum conditions}:

\begin{definition}
    \label{def:prequant}
     Let $(M,\omega)$ be a translation surface with singularities $\Sigma \subseteq M$ and $L \to M$ be a Heisenberg line bundle with wavelength $\hbar > 0$. An unbounded linear operator
    \[
    \mathcal{O} \colon L^2(M) \to \mathfrak{u}(\Gamma^2(L))
    \]
    defined on a dense subspace $\mathrm{Dom}(\mathcal{O}) \supseteq C_0^\infty(M \setminus \Sigma)$ is said to be a \textit{prequantization} of $(M,\omega)$ if it satisfies Dirac's quantum conditions, i.e., if the following conditions hold:
    \begin{enumerate}
        \item If $f \in L^2(M)$ is a constant function then $f \in \mathrm{Dom}(\mathcal{O})$ and $\mathcal{O}(f) \in \mathfrak{u}(\Gamma^2(L))$ is the corresponding multiplication operator on sections.
        \item If $f,g \in C_0^\infty(M \setminus \Sigma)$ then
        \[
        [\mathcal{O}(f),\mathcal{O}(g)] := \mathcal{O}(f) \circ\mathcal{O}(g) - \mathcal{O}(f) \circ\mathcal{O}(g) = -i \hbar \mathcal{O}(\{f,g\}).
        \]
    \end{enumerate}
\end{definition}

\begin{remark}
    Prequantizations are not canonical; a given translation surface $(M,\omega)$ could admit multiple prequantizations. Definition \ref{def:prequant} does not address the issue that a prequantization could be too large to be physically meaningful; the process of choosing a subspace of a prequantization whose spectral information is physically meaningful is called \textit{quantization}. We will not discuss this subject in this paper.
\end{remark}

\begin{remark}
    Condition (2) in Definition \ref{def:prequant} is the main reason behind considering non-flat line bundles in Definition \ref{def:hcb}; one cannot hope for such non-commutativity relation without the non-vanishing curvature assumption on the line bundle.
\end{remark}

For our purposes we will consider the following prequantization:

\begin{definition}\label{def:geomprequant}
    Let $(M,\omega)$ be a translation surface with singularities $\Sigma \subseteq M$ and $L \to M$ be a Heisenberg line bundle with wavelength $\hbar > 0$ and connection $\nabla$. The \textit{geometric prequantization} of $(M,\omega)$ with respect to $L$ is the operator
    \[
    \mathcal{O} \colon L^2(M) \to \mathfrak{u}(\Gamma^2(L))
    \]
    which to every constant function assigns the corresponding multiplication operator and which to every smooth compactly supported function $f \in C_0^\infty(M \setminus \Sigma)$ assigns the self-adjoint operator $O(f) \in \mathfrak{u}(\Gamma^2(L))$ acting on smooth sections $s \colon M \to E$ by
    \[
    O(f).s := -i\hbar \nabla_{X_f} s + fs.
    \]
\end{definition}

\begin{proposition}
    Let $(M,\omega)$ be a translation surface and $L \to M$ be a Heisenberg line bundle. Then, the geometric prequantization of $(M,\omega)$ with respect to $L$ satisfies Dirac's quantum conditions, i.e., it is indeed a prequantization.
\end{proposition}

\begin{proof}
    The first condition follows automatically. We now check the second condition. Let $\Sigma \subseteq M$ be the set of singularities of $(M,\omega)$. Consider functions $f,g \in C_0^\infty(M \setminus \Sigma)$. A direct computation using the Leibniz Rule and Definition \ref{def:poisson} shows that
    \begin{equation}
        \label{eq:B1}
        [\mathcal{O}(f) ,\mathcal{O}(g)] = -\hbar^2 (\nabla_{X_f} \nabla_{X_g}s - \nabla_{X_g} \nabla_{X_f} s) - 2 i \hbar \{f,g\}s.
    \end{equation}
    The definition of curvature gives
    \begin{equation}
        \label{eq:B2}
    \nabla_{X_f} \nabla_{X_g}s - \nabla_{X_g} \nabla_{X_f} s = -i\hbar^{-1}\eta(X_f,X_g) s - \nabla_{[X_f,X_g]} s.
    \end{equation}
    Again, by the definition of the Poisson bracket,
    \begin{equation}
        \label{eq:B3}
    \eta(X_f,X_g) = \{f,g\}.
    \end{equation}
    A direct computation using Cartan's identity shows that
    \begin{equation}
        \label{eq:B4}
    [X_f,X_g] = - X_{\{f,g\}}.
    \end{equation}
    Putting together \eqref{eq:B1}, \eqref{eq:B2}, \eqref{eq:B3}, and \eqref{eq:B4}, we conclude
    \[
    [\mathcal{O}(f) ,\mathcal{O}(g)] = -\hbar^2 \nabla_{X_{\{f,g\}}} s - i\hbar \{f,g\}s = -i \hbar \mathcal{O}(\{f,g\}). \qedhere
    \]
\end{proof}

\subsection{Prequantized dynamics.} We are not just interested in prequantizing the observables of a translation surface, but also the dynamics of its linear flows:

\begin{definition}\label{def:prequant:dynamics}
     Let $(M,\omega)$ be a translation surface, $L \to M$ be a Heisenberg line bundle, and $\mathcal{O} \colon L^2(M) \to \mathfrak{u}(\Gamma^2(L))$ be a prequantization of $(M,\omega)$. As usual, consider a linear vector field $W := aX+bY$ on $M$, with $(a,b) \neq (0,0)$, and recall that $\phi^W$ denotes the corresponding flow on $M$. A \textit{prequantization} of this flow with respect to $\mathcal{O}$ is a one-parameter subgroup $\mathcal{U}^W := \{\mathcal{U}^W_t\}_{t \in \mathbb{R}} \subseteq U(\Gamma^2(L))$ of bounded unitary operators
     \[
     \mathcal{U}^W_t \colon \Gamma^2(L) \to \Gamma^2(L)
     \]
     such that for every $t \in \mathbb{R}$ and every $f \in \mathrm{Dom}(\mathcal{O})$,
     \begin{equation}
     \label{eq:conj}
     (\mathcal{U}^W_t)^{-1} \circ \mathcal{O}(f) \circ \mathcal{U}^W_t = \mathcal{O}(f \circ \phi^W_t).
     \end{equation}
     The \textit{Hamiltonian} of such prequantization is the unique unbounded self-adjoint operator $\mathcal H^W$ generating $\mathcal{U}^W$, i.e., such that for every $t \in \mathbb{R}$,
     \[
     \mathcal{U}^W_t = e^{-it\mathcal {H^W}}.
     \]
\end{definition}

\paragraph*{\bf Geometric prequantizations} In the case of geometric prequantizations we consider the following construction:

\begin{definition}
    Let $(M,\omega)$ be a translation surface, $L \to M$ be a Heisenberg line bundle, and $W := aX+bY$ with $(a,b) \neq (0,0)$ be a linear vector field on $M$. Recall that $\smash{\Phi^{\wW}}$ denotes the flow on $L$ induced by $\phi^W$ via parallel transport. Consider the flow
\[
\Phi^{\wW,*} := \left\lbrace\smash{\Phi_t^{\wW,*}} \colon \Gamma^2(L) \to \Gamma^2(L) \right\rbrace_{t \in \mathbb{R}}
\]
defined for every $s \in \Gamma^2(L)$ and every $t \in \mathbb{R}$ by
\[
\Phi^{\wW,*}_t.s := \Phi_t^{\wW} \circ s \circ \phi_{-t}^{W}.
\]
We refer to $\Phi^{\wW,*}$ as the \textit{geometric prequantization} of $\phi^W$ with respect to $L$.
\end{definition}

\begin{proposition}
    \label{prop:prequan_dyn}
    Let $(M,\omega)$ be a translation surface, $L \to M$ be a Heisenberg line bundle, and $\mathcal{O} \colon L^2(M) \to \mathfrak{u}(\Gamma^2(L))$ be the geometric prequantization of $(M,\omega)$ with respect to $L$. Then, for every linear vector field $W := aX + bY$ on $M$ with $(a,b) \neq (0,0)$, the flow $\smash{\Phi^{\wW,*}}$ is a prequantization of $\phi_W$ with respect to $\mathcal{O}$ whose Hamiltonian is $\nabla_W$.
\end{proposition}

\begin{proof}
    The operators in $\smash{\Phi^{\wW,*}}$ are unitary because parallel transport preserves inner products. Furthermore, it follows directly from the definition of parallel transport that their generator is $\nabla_W$. Let $\Sigma \subseteq M$ be the set of singularities of $(M,\omega)$. To check \eqref{eq:conj} we fix $t \in \mathbb{R}$, $f \in C^\infty_0(M \setminus \Sigma)$, and a smooth section $s \colon M \to L$. Using the fact that the singular symplectic form $\eta := \frac{i}{2}(\omega \wedge \overline{\omega})$ is preserved by the flow $\phi^W_t$, one can show that, for every $t \in \mathbb{R}$, the Hamiltonian vector field of $f \circ \phi_t$ is given by
    \begin{equation}
        \label{eq:identity1}
        X_{f \circ \phi_t^W} = d(\phi_t^W)^{-1}.X_f .
    \end{equation}
    Furthermore, by general Riemannian geometry considerations, we have
    \begin{equation}
        \label{eq:identity2}
        \nabla_{X_f}(\Phi_t^{\wW} \circ s \circ \phi_{-t}^{W}) = \Phi_t^{\wW} \circ (\nabla_{d\phi_t^{-1} X_f} s) \circ \phi_{-t}^{W}.
    \end{equation}
    Putting together \eqref{eq:identity1} and \eqref{eq:identity2} we conclude
    \begin{align*}
    ((\mathcal{U}^W_t)^{-1} \circ \mathcal{O}(f) \circ \mathcal{U}^W_t).s &= - i \hbar \cdot \Phi_t^{\wW} \circ \nabla_{X_f}(\Phi_t^{\wW} \circ s \circ \phi_{-t}^{W}) \circ \phi_{t}^W + (f \circ \phi_t).s \\
    &= -i \hbar \cdot \nabla_{d(\phi_t^W)^{-1} X_f} s + (f \circ \phi_t^W).s \\
    &= -i \hbar \cdot \nabla_{X_{f \circ \phi_t^W}} s + (f \circ \phi_t^W).s\\
    &= \mathcal{O}(f \circ \phi_t^W).s. \qedhere
    \end{align*}
\end{proof}

\begin{remark}
    Proposition \ref{prop:prequan_dyn} leads to the following interpretation: Given a translation surface $(M,\omega)$, a Heisenberg line bundle $L \to M$, and a linear vector field $W := aX + bY$ on $M$ with $(a,b) \neq (0,0)$, the operator $\nabla_W$ on $\Gamma^2(L)$ is the prequantization of momentum along the linear flow $\phi^W$ on $M$. 
\end{remark}

%\francisco{Infinitesimal generator...}
%\jayadev{Do you plan to add to this?}

\section{Dynamics on wave functions}\label{sec:dynwave}

We now turn to explain how to apply the dynamical and spectral results proved in Sections \ref{sec:heisenberg:flow} and \ref{sec:spectral} to study the prequantized model introduced in Section \ref{sec:prequant}. For this purpose we first define the notion of \textit{wave transform}.

\subsection{Flows.} As above, let $(M,\omega)$ be a translation surface with $X$ and $Y$ the vertical and horizontal vector fields induced by $\omega$ on $M$. Given $W := aX+bY$ with $(a,b) \not = (0,0)$, the linear flow $\phi^W$ on $M$ induces a unitary flow $\phi^{W,*}$ on $L^2(M) := L^2(M, \eta)$ by
\[
\phi^{W,*}_t.f := f \circ \phi_{-t}^{W}, \quad \text{where } t \in \mathbb{R} \text{ and }  f \in L^2(M).
\]
Let $L \to M$ be a Heisenberg line bundle and recall that
\[
\Phi^{\wW} := \left\lbrace\smash{\Phi_t^{\wW}} \colon L \to L \right\rbrace_{t \in \mathbb{R}}
\]
is the flow induced by parallel transport along the orbits of $\phi^{W}$. Recall that $\Gamma^2(L)$ denotes the Hilbert space of square-integrable sections of $L$ and denote by
\[
\Phi^{\wW,*} := \left\lbrace\smash{\Phi_t^{\wW,*}} \colon \Gamma^2(L) \to \Gamma^2(L) \right\rbrace_{t \in \mathbb{R}}
\]
the natural unitary pullback flow given by
\[
\Phi^{\wW,*}_t.s := \Phi_t^{\wW} \circ s \circ \phi_{-t}^{W}, \quad \text{where }t \in \mathbb{R} \text{ and }  s \in \Gamma^2(L).
\]
Letting $C$ be the associated Heisenberg circle bundle on $M$ induced by $L$, we have a unitary flow $\Phi^{\wW,\circ}$ on $L^2(C) := L^2(C,\mu)$ given by
\[
\Phi^{\wW,\circ}_t.f := f \circ \Phi_{-t}^{\wW}, \quad \text{where } t \in \mathbb{R} \text{ and }  f \in L^2(C).
\]

\subsection{The wave transform.} Theorem \ref{thm:rel_mix_new} shows that the flows $\smash{\Phi^{\wW}}$ on $C$ are relatively mixing. This statement pertains to functions in $L^2(C)$ but from the point of view of quantum mechanics what we care about is the dynamics of the flow $\smash{\Phi^{\wW,*}}$ on the space of square-integrable sections $\Gamma^2(L)$, i.e., the prequantum space of wave functions. 

We now show how to connect these a priori distinct viewpoints via a family of operators we call \textit{wave transforms}. These operators map elements of $\Gamma^2(L^k)$, the space of square-integrable sections of the $k$-th tensor power $L^k$ of the line bundle $L$, to the eigenspaces $E_{-k}$ of the flow $\smash{\Phi^{\wW,\circ}}$ on $L^2(C)$. We recall that $L^2(M) \subseteq L^2(C)$ denotes the subspace of functions on $C$ constant along fibers and that $L^2_*(C) \subseteq L^2(C)$ denotes the subspace of functions on $C$ of zero integral along $\eta$-almost-every fiber of $C$, so that
\[
L^2(C) = L^2(M) \oplus L^2_*(C) \quad \text{and} \quad L^2(M)\perp L^2_*(C).
\]
In fact, we can write
\[
L^2(M)=E_0 \quad \text{and} \quad L^2_*(C) = \bigoplus_{k\in \Z\setminus\{0\}} E_k,
\]
where $E_k$ is the $k$-th eigenspace of the flow generated by the vector field
$\widehat{R}$. The hermitian product gives an identification of $L$ with the dual bundle $L^*$ and therefore an identification of the $k$-th tensor powers $\smash{(L^*)^k}$ and $\smash{L^k}$, which
can be equivalently identified via a canonical extension
$\langle \cdot, \cdot \rangle_k$ of the hermitian product to $L^k$. 

\begin{definition}
    Let $(M,\omega)$ be a translation surface and $L \to M$ be a Heisenberg line bundle. For every $k\in \Z\setminus \{0\}$ we define the \textit{$k$-th wave transform}
    \begin{align*}
    \mathcal{W}_k \colon \Gamma^2(L^k) &\to L^2(C)%\\
    %s &\mapsto \mathcal{W}_k[s]
    \end{align*}
    for every $s \in \Gamma^2(L^k)$ and every $(m,z) \in C$ by
    $$\mathcal{W}_k[s](m,z) := \langle s(x),z^{\otimes k} \rangle_k.$$
\end{definition}

The following is the main result of this section:

\begin{theorem}
    \label{theo:iso}
     Let $(M,\omega)$ be a translation surface, $W := aX+bY$ with $(a,b) \not = (0,0)$ be a vector field on $M$, and $L \to M$ be a Heisenberg line bundle. Then, for every $k\in \Z\setminus \{0\}$, the $k$-th wave transform is an operator
    \begin{align*}
    \mathcal{W}_k \colon \Gamma^2(L^k) &\to E_{-k}
    \end{align*}
    that is a unitary equivalence of Hilbert spaces intertwining the flows $\Phi^{\wW,*}$ and $\Phi^{\wW,\circ}$.
\end{theorem}

%\jayadev{I think we mean $\langle z, w \rangle = z \overline{w}$- should we just say this? One more notational issue- we want to integrate with respect to the area form induced by $\omega$, which is in fact $\eta$. But I think it may be worth introducing a different letter for the measure, maybe $m_{\eta}$, so we integrate $dm_{\eta}$? Also, earlier we just integrate $0$ to $2\pi$ in the circle factor, and I think it is probably nicest to stick to this.}

%\jayadev{First, I think we should stick with $L$ for the Hermitian line bundle, and not introduce $E$. I think with the definition we have, $\mathcal W(\Gamma^2(L)) \subset E_{-1},$ since \begin{align*}\mathcal W[s] \left(\phi^{\hat R}_t(x, z)\right) &= \mathcal W[s] (x, e^{2\pi i t} z) \\  &= \langle s(x), e^{2\pi it} z \rangle \\ &= e^{-2\pi it} \langle s(x), z \rangle \\ & = e^{-2\pi it} \mathcal W[s](x, z). \end{align*} A similar computation shows that we can define a map $\mathcal W_k$, with $\mathcal W_k(\Gamma^2(L)) \subset E_{-k},$ via $$\mathcal W_k[s](x, z) = \langle s(x), z^k \rangle.$$ }

\begin{proof}
    %We first note that $\mathcal W_k$ is well-defined (\jayadev{what should we say here? Do we want to use the identification with the dual bundle?}). 
    
    We first check that $\mathcal{W}_k$ preserves inner products. We use the fact that on any one-dimensional complex Hilbert space $V$ the following identity holds:
    \begin{equation*}
    \label{eq:identity}
    \langle z_1, z_3\rangle \overline{\langle z_2,z_3 \rangle} = \langle z_1,z_2 \rangle \langle z_3,z_3 \rangle, \quad \text{for all } z_1,z_2,z_3 \in V.
    \end{equation*}
    It follows that, denoting by $\rho_m$ the fiberwise measures on $C$, for every $s_1,s_2 \in \Gamma^2(L^k)$, 
    \begin{align*}
        \left\langle\mathcal{W}_k[s_1],\mathcal{W}_k[s_2]\right\rangle_{L^2(C)} &= \int_M \int_{L^k_m} \langle s_1(m),z^{\otimes k} \rangle_k \overline{\langle s_2(m),z^{\otimes k} \rangle_k} \thinspace d\rho_m(z) \thinspace d\eta(m) \\
        &= \int_M \int_{C^k_m} \langle s_1(m),s_2(m) \rangle_k \thinspace d\rho_m(z) \thinspace d\eta(m)\\
        &= \int_M \langle s_1(m),s_2(m) \rangle_k \thinspace d\eta(m)\\
        &= \langle s_1,s_2\rangle_{\Gamma^2(L^k)}.
    \end{align*}

    The inclusion the image of $\mathcal{W}_k$ in $E_{-k}$ can be derived as follows: For every $s \in \Gamma^2(L^k)$, every $t \in \mathbb{R}$, and every $(m,z) \in C$, writing $z:= e^{2\pi i \rho}$ with $\rho \in \mathbb{R}$,
$$
\begin{aligned}
(\mathcal{W}_k(s) \circ \Phi^{\wR}_t) (m,z) &=\langle s(m), (e^{2\pi it} e^{2\pi i\rho})^k\rangle_k \\ &= e^{-2\pi ikt} \langle s(m), e^{2\pi i \rho k}\rangle_k \\
&= e^{-2\pi ikt}\mathcal{W}_k(s)(m,z).
\end{aligned}
$$

We now prove that $\mathcal{W}_k$ surjects onto $E_{-k}$. Fix a square-integrable section $s_0 \colon M \to L$. Given any $f\in E_{-k}$, let $s_f \colon M \to L^k$ be the unique section determined by the condition
$$
\langle s_f(m), s_0^{\otimes k}(m) \rangle_k = f(s_0(m)) \, \quad \text{ for $\eta$-almost-every } m \in M \,.
$$
We claim that $\mathcal{W}_k[s_f] =f$. Indeed, given $(m,z) \in C$ we write $z = e^{2 \pi i \rho} s_0(m)$ for some $\rho \in \mathbb{R}$. Since $f \in E_{-k}$ we have 
$$
\begin{aligned}
\mathcal{W}_k[s_f](m,z)
&= \langle s_f(m), e^{2\pi i k\rho} s_0^{\otimes k}(m) \rangle_k \\
&= e^{-2 \pi ik \rho} \langle s_f(m), s_0^{\otimes k}(m) \rangle_k \\
&= e^{-2 \pi ik \rho} f(s_0(m)) \\
&= f(e^{2 \pi i \rho}s_0(m))\\
&= f(m,z).
\end{aligned}$$

    Finally, to show that $\mathcal{W}_k$ intertwines the flows $\Phi^{\wW,*}$ and $\Phi^{\wW,\circ}$ we use the fact that the parallel transport induced by $\nabla$ preserves inner products. Indeed, for any $s \in \Gamma^2(L^k)$, any $t \in \mathbb{R}$, and any $(m,z) \in C$, we have
    \begin{align*}
    \mathcal{W}_k\left[\Phi^{\wW,*}_t.s \right](m, z) &= \left\langle \left(\Phi^{\wW,*}_t.s\right)(m),z^{\otimes k} \right\rangle_k \\
    &= \left\langle \Phi_t^{\wW}.s(\phi_{-t}^{W}.m ),z^{\otimes k}\right\rangle_k\\
    &= \left\langle s(\phi_{-t}^{W}.m), \Phi_{-t}^{\wW}(z^{\otimes k})\right\rangle_k\\
    &= \mathcal{W}_k[s]\left(\Phi_{-t}^{\wW}(m,z)\right)\\ 
    &= \left(\Phi^{\wW,\circ}_t.\mathcal{W}_k[s]\right)(m,z). \qedhere
    \end{align*}
\end{proof}

\subsection{Decay of correlations.} As a direct consequence of Theorems \ref{thm:rel_mix_new} and \ref{theo:iso} we deduce that prequantum wave functions become uncorrelated under time evolution:

\begin{corollary}
    \label{cor:decay}
    Let $(M,\omega)$ be a translation surface, $W := aX+bY$ with $(a,b) \not = (0,0)$ be a vector field on $M$, and $L \to M$ be a Heisenberg line bundle. Then, for $s_1,s_2 \in \Gamma^2(L)$,
    \[
    \lim_{t \to \infty} \left\langle \Phi_t^{\wW,*}.s_1,s_2 \right\rangle_{\Gamma^2(L)} = 0.
    \]
\end{corollary}

\subsection{Spectra.} As a direct consequence of Theorems \ref{prop:spectra} and \ref{theo:iso} we deduce information about the spectrum of the prequantized momentum operators of translation flows:

\begin{corollary}
    Let $(M,\omega)$ be a translation surface, $W := aX+bY$ with $(a,b) \not = (0,0)$ be a vector field on $M$, and $L \to M$ be a Heisenberg line bundle. Then, the prequantized momentum operator $\nabla_W$ on $\Gamma^2(L)$ has absolutely continuous spectrum, and, if the translation flow $\phi^{W}$ on $M$ is aperiodic, then the spectrum has no gaps, i.e., the maximal spectral type of $\nabla_W$ on $\Gamma^2(L)$ has full support.
\end{corollary}

\subsection{Consequences.} Let us now explore some consequences of Corollary \ref{cor:decay}. First, we observe this result can be directly extended to finite dimensional orthogonal projections:

\begin{corollary}
    \label{cor:proj}
    Let $(M,\omega)$ be a translation surface, $W := aX+bY$ with $(a,b) \not = (0,0)$ be a vector field on $M$,  $L \to M$ be a Heisenberg line bundle, and $P \colon \Gamma^2(L) \to \Gamma^2(L)$ be an orthogonal projection to a finite dimensional subspace. Then, for every $s \in \Gamma^2(L)$,
    \[
    \lim_{t \to \infty} P\left(\Phi_t^{\wW,*}.s\right) = 0.
    \]
\end{corollary}

Now let $A$ be an unbounded self-adjoint operator on $\Gamma^2(L)$ with domain $\mathrm{Dom}(A)$ and denote convergence of operators on $\Gamma^2(L)$ in the strong topology by $s - \lim$. By the Spectral Theorem (see, for example, \cite[Theorem 7.20]{Hall}), $A$ corresponds to a unique spectral family $\{P_\lambda\}_{\lambda \in \mathbb{R}}$. More concretely, each $P_\lambda$ is an orthogonal projection to a closed subspace of $\Gamma^2(L)$ and the family satisfies:
\begin{enumerate}
    \item Non-decreasing, i.e., $P_\lambda P_{\lambda'} = P_{\min\{\lambda,\lambda'\}}$ for all $\lambda,\lambda' \in \mathbb{R}$.
    \item Strong right continuity, i.e., $P_\lambda = s - \lim_{\epsilon \to 0^+} P_{\lambda+\epsilon}$ for all $\lambda \in \mathbb{R}$.
    \item $s - \lim_{\lambda \to -\infty} P_\lambda =\mathbf{0}$ and $s - \lim_{\lambda \to +\infty} P_\lambda =\mathbf{1}$.
\end{enumerate}
In this context, for every $s \in \Gamma^2(L)$ one considers the function
\begin{align*}
F_s \colon \mathbb{R} &\to \mathbb{R} \\
\lambda &\mapsto \langle P_\lambda s, s\rangle_{\Gamma^2(L)}
\end{align*}
and the corresponding Stieltjes measure $m_s$ on $\mathbb{R}$, i.e., such that
\[
m_s((a,b]) := F_s(b) - F_s(a) \quad \text{for all } a <b \in \mathbb{R}.
\]
Then, $\mathrm{Dom}(A)$ can be described as
\[
\mathrm{Dom}(A) := \left\lbrace s \in \Gamma^2(L) \colon \int_\mathbb{R} \lambda^2 \thinspace dm_s(\lambda) < \infty\right\rbrace
\]
and $A$ satisfies
\begin{equation}
\label{eq:avg}
\langle As,s\rangle_{\Gamma^2(L)} = \int_\mathbb{R} \lambda \thinspace dm_s(\lambda) \quad \text{for every } s \in \Gamma^2(L).
\end{equation}
If we consider $A$ as a quantum observable of our system, then the measure $m_s(V) / \|s\|$ on $\mathbb{R}$ describes the probability distribution of the observable represented by $A$, and, in particular, the expected value of this quantity can be computed using the formula in \eqref{eq:avg}.

As usual, we say that $\lambda \in \mathbb{R}$ is in the point spectrum of $A$ if it admits an eigenfunction, i.e., if there exists $0 \neq s \in \mathrm{Dom}(A)$ such that $As = \lambda s$. In terms of the corresponding spectral family, this is equivalent to the condition that $\mathrm{Ran}(P(\{\lambda\})) \neq \{0\}$, where 
\[
P(\{\lambda\}) := P_\lambda - \left(s -\lim_{\epsilon\to 0^+} P_{\lambda-\epsilon}\right).
\]
Furthermore,
\[
s \in \mathrm{Ran}(P(\{\lambda\})) \Longleftrightarrow s \in \mathrm{Dom}(A) \text{ and } As = \lambda s.
\]
Given any $s \in \Gamma^2(L)$ the Stieltjes measure $m_s$ has a (potentially trivial) atom at $\lambda$ of size
\[
m_s(\{\lambda\}) = \|P(\{\lambda\}) s\|.
\]

From this description and Corollary \ref{cor:proj} we deduce that, if $\lambda \in \mathbb{R}$ is not an eigenvalue of $A$ or is an eigenvalue with finite dimensional eigenspace, then the potential atom at $\lambda$ of the Stieltjes measure of an arbitrary wave function dissipates under time evolution:

\begin{corollary}
    \label{cor:spec}
    Let $(M,\omega)$ be a translation surface, $W := aX+bY$ with $(a,b) \not = (0,0)$ be a vector field on $M$, and $L \to M$ be a Heisenberg line bundle. Suppose that $\lambda \in \mathbb{R}$ is such that $\mathrm{Ran}(P(\{\lambda\})) \subseteq \Gamma^2(L)$ is finite dimensional. Then, for every $s \in \Gamma^2(L)$,
    \[
    \lim_{t \to \infty} m_{\Phi_t^{\wW,*}.s}(\{\lambda\}) = 0.
    \]
\end{corollary}

%\section{Prequantization of translation flows}

\section{Affine skew products}\label{sec:affineskewiet}

We now explain how certain families of affine skew products over interval exchange transformations arise as first return maps of Heisenberg translation flows to appropriately chosen cross sections; we explicitly describe the set of skew products that arise in this way. As an application of the results derived in previous sections we describe ergodic properties of such skew products, proving Theorem~\ref{theorem:admissible} and Corollary~\ref{cor:admissible}.

\subsection{Interval Exchange Transformations.} An interval exchange transformation (IET) is a bijective piecewise isometry $T \colon I \to I$ of an interval $I \subseteq \mathbb{R}$; without loss of generality we assume the left endpoint of $I$ is $0$. Such map is specified by the following data:
\begin{itemize}
    \item A partition of $I$ into subintervals $\{I_\alpha\}_{\alpha \in A}$ with $A$ a finite alphabet;
    \item A pair of bijections $\pi_i \colon A \to \{1,\dots,d\}$ with $i \in \{0,1\}$ specifying the order of the subintervals before and after the map $T$ is applied;
    \item A positive length vector $\lambda := (\lambda_\alpha)_{\alpha \in A} \in \mathbb{R}_+^A$ with $\sum_{\alpha \in A} \lambda_\alpha = |I|$ specifying the length of each of the subintervals.
\end{itemize}
Given the pair $\pi := (\pi_0,\pi_1)$, denote by $\Omega_\pi \colon \mathbb{R}^A \to \R^A$ the linear map defined by
\[
(\Omega_\pi)_{\alpha,\beta} := \left\lbrace \begin{array}{l l}
     +1& \text{if $\pi_1(\alpha) > \pi_1(\beta)$ and $\pi_0(\alpha)< \pi_0(\beta)$;} \\
     -1& \text{if $\pi_1(\alpha) < \pi_1(\beta)$ and $\pi_0(\alpha)> \pi_0(\beta)$;} \\
     0 & \text{otherwise.}
\end{array} \right.
\]
Define $w := (w_\alpha)_{\alpha \in A} \in \R^A$ by
\[
w_\alpha := -\sum_{\pi_0(\beta) < \pi_0(\alpha)} \tau_\beta + \sum_{\pi_1(\beta) < \pi_1(\alpha)} \tau_\beta = \Omega_\pi(\lambda)_\alpha.
\]
Then, the corresponding IET $T \colon I \to I$ is given by
\[
T(x) = x + w_\alpha, \quad \text{if } x \in I_\alpha.
\]
The monodromy invariant of $T$ is the permutation $p \colon \{1, \dots, d \} \to \{1,\dots,d\}$ given by
\[
p = \pi_1 \circ \pi_0^{-1}.
\]
We say $\pi$ is irreducible if $p$ does not decompose into disjoint permutations.

\subsection{Zippered Rectangles.} Consider an interval exchange $T \colon I \to I$ described by the data specified above. To simplify the discussion we will assume $\pi$ is irreducible. Denote by $T_\pi^+$ the convex cone of vectors $\tau := (\tau_\alpha)_{\alpha \in A} \in \R^A$ such that
\[
\sum_{\pi_0(\alpha) \leq k} \tau_\alpha > 0 \quad \text{and} \quad \sum_{\pi_1(\alpha) \leq k} \tau_\alpha < 0 \quad \text{for all } 1 \leq k \leq d-1.
\]
Given $\tau \in T_\pi^+$, define $h := (h_\alpha)_{\alpha \in A} \in \R^A$ by
\[
h_\alpha := \sum_{\pi_0(\beta) < \pi_0(\alpha)} \tau_\beta - \sum_{\pi_1(\beta) < \pi_1(\alpha)} \tau_\beta = -\Omega_\pi(\tau)_\alpha.
\]
Notice that, as $\tau \in T_\pi^+$, $h_\alpha > 0$ for all $\alpha \in A$. Denote $H_\pi^+ := -\Omega_\pi(T_\pi^+).$

Suppose $\tau \in T_\pi^+$. For each $\alpha \in A$ consider the rectangles of $\mathbb{R}^2$ defined by
\begin{align*}
    R_\alpha^0 &:=\left( \sum_{\pi_0(\beta) < \pi_0(\alpha) }\lambda_\beta, \sum_{\pi_0(\beta) \leq \pi_0(\alpha) }\lambda_\beta \right) \times [0,h_\alpha],\\
    R_\alpha^1 &:=\left( \sum_{\pi_1(\beta) < \pi_1(\alpha) }\lambda_\beta, \sum_{\pi_1(\beta) \leq \pi_1(\alpha) }\lambda_\beta \right) \times [-h_\alpha,0].
\end{align*}
Consider also the vertical segments of $\mathbb{R}^2$ given by
\begin{align*}
    S_\alpha^0 &:= \left\lbrace \sum_{\pi_0(\beta) \leq \pi_0(\alpha)} \lambda_\beta\right\rbrace \times \left[0, \sum_{\pi_0(\beta) \leq \pi_0(\alpha)} \tau_\beta \right],\\
    S_\alpha^0 &:= \left\lbrace \sum_{\pi_1(\beta) \leq \pi_1(\alpha)} \lambda_\beta\right\rbrace \times \left[\sum_{\pi_1(\beta) \leq \pi_1(\alpha)} \tau_\beta, 0 \right].
\end{align*}
Here, by convention, if $\alpha(0) := \pi_0^{-1}(d) \in \mathcal{A}$ and $\alpha(1) := \pi_1^{-1}(d) \in \mathcal{A}$, then
\[
S_{\alpha(0)}^0 = S_{\alpha(1)}^1 := \left\lbrace \sum_{\beta \in \mathcal{A}} \lambda_\beta\right\rbrace \times \left[0, \sum_{\beta \in \mathcal{A}} \tau_\beta \right];
\]
this segment can either be above of below the horizontal axis.

Consider the translation surface $M = M(\pi,\lambda,\tau,h)$ obtained as the quotient of
\[
\Xi :=\bigcup_{\alpha \in A} \bigcup_{\epsilon \in \{0,1\}} R_\alpha^\epsilon \cup S_{\alpha}^\epsilon
\]
by the equivalence relation $\sim$ we now describe. First, identify each rectangle $R_\alpha^0$ with its counterpart $R_\alpha^1$ through the translation
\[
(x,y) \mapsto (x + w_\alpha, z-h_\alpha).
\]
Suppose now that the sum of all the $\tau_\beta$ is positive. Then, for $\alpha = \alpha(1)$ consider
\[
\widetilde{S} := \left\lbrace \sum_{\pi_0(\beta) \leq \pi_0(\alpha)} \lambda_\beta \right\rbrace \times  \left[h_{\alpha}, \sum_{\pi_0(\beta) \leq \pi_0(\alpha)} \tau_\beta \right],
\]
i.e., $\widetilde{S}$ is a segment at the top of $S_\alpha^0$. Analogously, if the sum of all the $\tau_\beta$ is negative, for $\alpha = \alpha(0)$ consider instead the segment at the bottom of $S_\alpha^1$ given by
\[
\widetilde{S} := \left\lbrace \sum_{\pi_1(\beta) \leq \pi_1(\alpha)} \lambda_\beta \right\rbrace \times  \left[\sum_{\pi_1(\beta) \leq \pi_1(\alpha)} \tau_\beta ,-h_{\alpha}\right].
\]
Then identify $\widetilde{S}$ with $S_{\alpha(0)}^0 = S_{\alpha(1)}^1$ by translation. This completes the definition of $\sim$. See Figure \ref{fig:zip} for an example of this construction.

\begin{figure}[ht]
\centering
\includegraphics[scale=1.5]{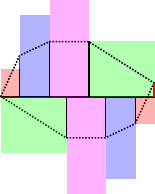}
\caption{Example of the zippered rectangle construction.}
\label{fig:zip}
\end{figure}

\label{sec:skew}

\subsection{Affine skew products.} A piecewise affine skew product over an interval exchange transformation $T \colon I \to I$ as above is specified by the following data:
\begin{itemize}
    \item A positive vector $h := (h_\alpha)_{\alpha \in A} \in \mathbb{R}_+^\mathcal{A}$ encoding the derivative of the skew product transformation above each subinterval;
    \item A vector $b := (b_\alpha)_{\alpha \in A} \in (\mathbb{R}/\mathbb{Z})^\mathcal{A}$ specifying the translational part of the skew product transformation above each subinterval.
\end{itemize}
More specifically, if $\partial I_\alpha \in I$ denotes the left endpoint of the interval $I_\alpha$ for $\alpha \in A$, then the corresponding skew product $\smash{\widehat{T}} \colon I \times \mathbb{R}/\mathbb{Z} \to \mathbb \colon I \times \mathbb{R}/\mathbb{Z}$ is given by
\begin{equation}
\label{eq:skew}
\smash{\widehat{T}}(x,\rho) = (x + w_\alpha, \thinspace \rho+h_\alpha(x-\partial I_\alpha) + b_\alpha), \quad \text{if } x \in I_\alpha.
\end{equation}

Let $(M,\omega)$ be a translation surface with vertical vector field $X$ and $C \to M$ be a Heisenberg circle bundle. Recall that 
\[
\Phi^{\wX} := \left\lbrace{\Phi_t^{\wX}} \colon C \to C \right\rbrace_{t \in \mathbb{R}}
\]
denotes the flow induced by parallel transport along the orbits of $X$. In this context we have the following characterization of first return maps:

\begin{proposition}
    \label{prop:skew1}
    Let $(M,\omega)$ be a translation surface with vertical vector field $X$ and $C \to M$ be a Heisenberg circle bundle. Consider a union of disjoint non-vertical straight line segments $I \subseteq M$ and denote by $P \subseteq E$ to be the set of fibers above points in $I$. Identify $P$ with $I \times \mathbb{R}/\mathbb{Z}$ by parallel transport. Then, the first return map of the flow $\smash{\Phi^{\wX}}$ on $E$ to $P$ is an affine skew product over an interval exchange transformation.
\end{proposition}

\begin{proof}
    Without loss of generality we can assume $I$ is a single horizontal segment. It is a well known fact that the first return map of the vertical flow on $M$ to $I$ is an interval exchange transformation $T \colon I \to I$. Furthermore, $M = M(\pi,\lambda,\tau,h)$ can be obtained via the zippered rectangle construction. Consider the notation introduced above. Given $\alpha \in A$ let $c_\alpha := h_\alpha > 0$ and $b_\alpha \in \mathbb{R}/\mathbb{Z}$ be the second coordinate of the result of parallel transporting $(\partial I_\alpha,0) \in I \times \mathbb{R}/\mathbb{Z} = P$ along the left side of the corresponding rectangle until its first return to $P$. Then, the curvature formula \eqref{eq:curv} guarantees the first return map under consideration is given precisely by the formula in \eqref{eq:skew}; see Figure \ref{fig:zip2}.
\end{proof}

\begin{figure}[ht]
\centering
\includegraphics[scale=1.5]{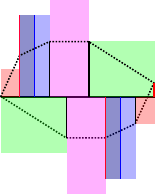}
\caption{The proof of Proposition \ref{prop:skew1}. The shaded rectangle contributes to the twist of the second coordinate of the skew product transformation via the curvature formula \eqref{eq:curv}.}
\label{fig:zip2}
\end{figure}

\begin{remark}
Below, in Theorem \ref{theo:skew2}, we completely characterize the set of affine skew products over an interval exchange transformation that can arise via Proposition \ref{prop:skew1}.
\end{remark}

\subsection{Homological Interpretation.} Consider an interval exchange $T \colon I \to I$ described by the data specified above. Consider the permutation $\sigma \colon \{0,\dots,d\} \to \{0,\dots,d\}$ given by
\[
\sigma(j)  := \left\lbrace \begin{array}{l l}
    p^{-1}(1)-1 & \text{if } j =0; \\
    d & \text{if } j = p^{-1}(d);\\
    p^{-1}(p(j)+1)-1 & \text{otherwise.}
\end{array} \right.
\]
The orbits of $\sigma$ encode the singularities of the suspended surface. Furthermore, for each orbit $\mathcal{O}$ of $\sigma$ not containing $0$ consider the vector $\lambda(\mathcal{O}) \in \R^A$ given by
\[
\lambda(\mathcal{O})_\alpha := \chi_\mathcal{O}(\pi_0(\alpha)) - \chi_\mathcal{O}(\pi_0(\alpha)-1)= \left\lbrace \begin{array}{ll}
    +1 & \text{if $\pi_0(j) \in \mathcal{O}$ but $\pi_0(j-1) \notin \mathcal{O}$}; \\
    -1 & \text{if $\pi_0(j) \notin \mathcal{O}$ but $\pi_0(j-1) \in \mathcal{O}$}; \\
    0 & \text{otherwise}.
\end{array} \right.
\]
Then, a basis of $\mathrm{Ker} \thinspace \Omega_\pi \subseteq \R^A$ is given by
\[
\{\lambda(\mathcal{O}) \colon \mathcal{O} \text{ is an orbit of $\sigma$ not containing $0$}\}.
\]

Consider a suspension surface $M = M(\pi,\lambda,\tau,h)$ obtained as above. Identify the horizontal segment $\gamma := I \times \{0\} \subseteq M$ with the interval $I$. For every $\alpha \in A$ consider the closed curve $v_\alpha$ obtained by following the vertical segment along the left side of $R_\alpha^0 \sim R_\alpha^1$ and then closing up along $\sigma$. Then, the following holds; see \cite{VianaIET} for a proof.

\begin{lemma}
    \label{lemma:hom}
    In the context above, the linear map
    \[
    \begin{array}{cccc}
        F \colon & \R^A &\to & H_1(M;\mathbb{R}) \\
        & c:=(c_\alpha)_{\alpha \in A} &  \mapsto &\sum_{\alpha \in A} c_\alpha[v_\alpha]
    \end{array}
    \]
    is surjective with kernel equal to $\mathrm{Ker} \thinspace \Omega_\pi \subseteq \R^A$.
\end{lemma}

\begin{remark}
Rather than using Lemma \ref{lemma:hom} directly, we will follow the arguments in its proof together with Proposition \ref{prop:skew1} to derive a proof of Theorem \ref{theo:skew2} below.
\end{remark}

\subsection{Admissible skew products.}\label{sec:admissible} We now introduce the explicit family of affine skew products over interval exchange transformations that can arise via Proposition \ref{prop:skew1}.

\begin{definition}
    Let $T \colon I \to I$ be an interval exchange transformation described by the data specified above. Denote by $\mathcal{A}_T$ the set of affine skew products over $T$ specified as in \eqref{eq:skew} by vectors $h := (h_\alpha)_{\alpha \in A} \in \mathbb{R}_+^\mathcal{A}$ and $b := (b_\alpha)_{\alpha \in A} \in (\mathbb{R}/\mathbb{Z})^\mathcal{A}$ satisfying:
\begin{itemize}
    \item $h \in H_\pi^+ :=H_T$;
    \item $\sum_{\alpha \in A} \lambda_\alpha h_\alpha \in 2\pi\mathbb{Z}$;
    \item For all orbits $\mathcal{O}$ of $\sigma$ not containing $0$,
    \[
    \sum_{k \in \mathcal{O}} b_{\pi_0^{-1}(k)} - b_{\pi_0^{-1}(k+1)} = -\sum_{k \in \mathcal{O}} \lambda_{\pi_0^{-1}(k)} h_{\pi_0^{-1}(k)} \quad \mathrm{mod} \ \mathbb{Z},
    \]
    where by definition we let $\smash{b_{\pi_0^{-1}(d+1)} := 0}$.
\end{itemize}
We refer to $\mathcal{A}_T$ as the set of \textit{admissible affine skew products} over $T$.
\end{definition}

\begin{remark}
    Notice that if $\mathcal{O}$ has a single orbit, which necessarily contains $0$, then the third bullet point in the definition above imposes no additional constraints. This condition is equivalent to any suspension surface $M = M(\pi,\lambda,\tau,h)$ having a single singularity.
\end{remark}

In this context we prove the following characterization:

\begin{theorem}
    \label{theo:skew2}
    Let $T \colon I \to I$ be an interval exchange transformation. Then $\mathcal{A}_T$ is precisely the set of skew products over $T$ that arise via Proposition \ref{prop:skew1}.
\end{theorem}

\begin{proof}
    Let $M = M(\pi,\lambda,\tau,h)$ be a suspension surface obtained by applying the zippered rectangle construction to $T$ and $C \to M$ be a Heisenberg circle bundle. Denote by $$\smash{\widehat{T}} \colon I \times \mathbb{R}/\mathbb{Z} \to \mathbb \colon I \times \mathbb{R}/\mathbb{Z}$$ the affine skew product obtained via Proposition \ref{prop:skew1} over the horizontal segment $\gamma = I \times \{0\} \subseteq M$. From the proof of Proposition \ref{prop:skew1} we see that the parameters $h_\alpha$ of the skew product indeed correspond to the heights of the rectangles $R_\alpha^0 \sim R_\alpha^1$. In particular, $h \in H_\pi^+$. Furthermore, Weil's integrality condition \eqref{eq:weil} translates to
    \[
    \mathrm{Area}(M) = \sum_{\alpha \in A} \lambda_\alpha h_\alpha \in 2\pi\mathbb{Z}.
    \]

    Now, as in the proof of Proposition \ref{prop:skew1}, for every $\alpha \in A$ let $b_\alpha \in \mathbb{R}/\mathbb{Z}$ be the second coordinate of the result of parallel transporting $(\partial I_\alpha,0) \in I \times \mathbb{R}/\mathbb{Z} = P$ along the vertical flow of $M$ until its first return to $P$. Consider the map from the space of Heisenberg circle bundles $C \to M$ to $(\mathbb{R}/\mathbb{Z})^\mathcal{A}$ that records the vector $b := (b_\alpha)_{\alpha \in A}$. We claim that, for all orbits $\mathcal{O}$ of $\sigma$ not containing $0$,
    \[
    \sum_{k \in \mathcal{O}} b_{\pi_0^{-1}(k)} - b_{\pi_0^{-1}(k+1)} = \sum_{k \in \mathcal{O}} \lambda_{\pi_0^{-1}(k)} h_{\pi_0^{-1}(k)} \quad \mathrm{mod} \ \mathbb{Z}.
    \]
    These equations are clearly linearly independent. By Theorem \ref{theo:prequant} and Lemma \ref{lemma:hom}, it remains to show that these equation are indeed satisfied by the parameters $b$.

    Fix an orbit $\mathcal{O}$ of $\sigma$ not containing $0$. For simplicity we will assume $d \notin \mathcal{O}$, but the case $d \in \mathcal{O}$ follows by similar arguments. On can show that, as $1$-chains,
    \[
    \sum_{\alpha \in A} \lambda(\mathcal{O})_\alpha v_\alpha= \sum_{k \in \mathcal{O}} \left(v_{\pi_0^{-1}(k)} - v_{\pi_0^{-1}(k+1)}\right)=J - \sum_{k \in \mathcal{O}} \partial R_{\pi_0^{-1}(k)}^0,
    \]
    where $J$ is a $1$-chain in $\gamma = I \times \{0\} \subseteq M$ and the rectangles on the right hand side are outwards oriented. Computing the holonomies on both sides of this equality and using the curvature formula \eqref{eq:curv} yields the desired identity
     \[
    \sum_{k \in \mathcal{O}} b_{\pi_0^{-1}(k)} - b_{\pi_0^{-1}(k+1)} = -\sum_{k \in \mathcal{O}} \lambda_{\pi_0^{-1}(k)} c_{\pi_0^{-1}(k)} \quad \mathrm{mod} \ \mathbb{Z}. \qedhere
    \]
\end{proof}

\subsection{Dynamics and ergodic theory.} We now explore some consequences of Theorem \ref{theo:skew2} on the dynamics of affine skew products over interval exchange transformations. Recall that if $(M,\omega)$ is a translation surface, we denote by $\eta := \frac{i}{2} (\omega \wedge\overline{\omega})$ the corresponding Euclidean area form; we also denote by $\eta$ the corresponding Euclidean measure. Recall also that if $C \to M$ is a Heisenberg circle bundle, we denote by $\nu$ the measure on $C$ that disintegrates as Lebesgue on the fibers and as $\eta$ on the base.

\begin{corollary}
    \label{cor:erg}
    Let $T \colon I \to I$ be and interval exchange transformation and $\smash{\widehat{T}} \colon I \times \mathbb{R}/\mathbb{Z} \to I \times \mathbb{R}/\mathbb{Z}$ be an admissible affine skew product over $T$. Then, if $T$ is any of the following,
\begin{itemize}
    \item minimal,
    \item ergodic with respect to Lebesgue,
    \item uniquely ergodic,
\end{itemize}
then $\smash{\widehat{T}}$ is, respectively,
\begin{itemize}
    \item minimal
    \item ergodic with respect to $Lebesgue$
    \item uniquely ergodic.
\end{itemize}
\end{corollary}

\begin{proof}
    Using Theorem \ref{theo:skew2}, let $M = M(\pi,\lambda,\tau,h)$ be as in Proposition \ref{prop:skew1}, so that the first return map to the cross section $P = I \times \mathbb{R}/\mathbb{Z}$ of the flow $\smash{\Phi^{\widehat{X}}}$ on $C$ is given by $\smash{\widehat{T}}$. Notice that $T$ is minimal/Lebesgue-ergodic/uniquely-ergodic if and only the vertical flow $\phi^X$ of $M$ is minimal/$\eta$-ergodic/uniquely-ergodic. Furthermore, by Corollary \ref{cor:ethry}, this is equivalent to the flow $\smash{\Phi^{\widehat{X}}}$ on $C$ being minimal/$\nu$-ergodic/uniquely-ergodic. 
    
    The equivalence of the minimality properties is then a direct consequence of the fact that $\smash{\widehat{T}}$-orbits can be obtained as the intersection of $\smash{\Phi^{\widehat{X}}}$-orbits with $P$. Now recall that, for any dynamical system, ergodic measures are exactly the extreme points of the simplex of invariant measures. Thus, the equivalence of the ergodicity and unique-ergodicity properties can be obtained from the following fact: given a $\smash{\widehat{T}}$-invariant measure $\mu$ on $P$, its product with Lebesgue along the $\smash{\Phi^{\widehat{X}}}$-flow-direction is a $\smash{\Phi^{\widehat{X}}}$-invariant measure on $C$; furthermore, this assignment gives an affine bijection of invariant measures. 
\end{proof}

\begin{question}
Does Corollary \ref{cor:erg} hold without the admissibility hypothesis on the affine skew product or is such condition also necessary?
\end{question}

\paragraph*{\bf Constant slopes.} Exporting mixing properties of a flow to its first return maps is a much more delicate question. Nevertheless, in the case where the return times are constant, this is possible, as we now explain. Given an interval $I \subseteq \mathbb{R}$, decompose
\[
L^2(I \times \mathbb{R}/\mathbb{Z}) = L^2(I) \oplus L^2_*(I \times \mathbb{R}/\mathbb{Z}), \quad \text{with } L^2(I) \perp L^2_*(I \times \mathbb{R}/\mathbb{Z}),
\]
where $L^2(I)$ denotes the space of functions depending only on the first coordinate and $L^2_*(I \times \mathbb{R}/\mathbb{Z})$ denotes the space of functions whose fiberwise Lebesgue intergral along the second coordinate vanishes for Lebesgue almost every fiber. A measurable map $\smash{\widehat{T}} \colon I \times \mathbb{R}/\mathbb{Z} \to I \times \mathbb{R}/\mathbb{Z}$ is said to be \textit{relatively mixing} if
\[
\lim_{n \to \infty} \langle f \circ \widehat{T}^n,g\rangle_{L^2(I \times \mathbb{R}/\mathbb{Z}) } = 0, \quad \text{for every }f \in L^2_*(I \times \mathbb{R}/\mathbb{Z}) \text{ and } g \in L^2(I \times \mathbb{R}/\mathbb{Z}).
\]

\begin{corollary}
    \label{cor:erg2}
    Let $T \colon I \to I$ be and interval exchange transformation and $\smash{\widehat{T}} \colon I \times \mathbb{R}/\mathbb{Z} \to I \times \mathbb{R}/\mathbb{Z}$ be an admissible affine skew product over $T$ as in \eqref{eq:skew} with $h \in \R^A_+$ a constant vector. Then, $\smash{\widehat{T}}$ is relatively mixing.
\end{corollary}

\begin{proof}
    Using Theorem \ref{theo:skew2}, let $M = M(\pi,\lambda,\tau,h)$ be as in Proposition \ref{prop:skew1} so that the first return map to the cross section $P = I \times \mathbb{R}/\mathbb{Z}$ of the flow $\smash{\Phi^{\widehat{X}}}$ on $C$ is given by $\smash{\widehat{T}}$. By assumption we have that $h \in \mathbb{R}_+^\mathcal{A}$ is constant, say with value $c > 0$. In particular, $\smash{\Phi^{\widehat{X}}}_{cn}|_P = \smash{\widehat{T}^n}$ for every $n \in \mathbb{Z}$. Given $f,g\in L^2(I \times \mathbb{R}/\mathbb{Z})$ one can use parallel transport to thicken these functions along the $\smash{\Phi^{\widehat{X}}}$ flow direction to obtain $\smash{\widehat{f}}, \smash{\widehat{g}} \in L^2(C)$ such that
    \[
    \left\langle f \circ \widehat{T}^n,g\right\rangle_{L^2(I \times \mathbb{R}/\mathbb{Z}) } = \left\langle \smash{\widehat{f}} \circ \Phi^X_{cn},\smash{\widehat{g}}\right\rangle_{L^2(C)}, \quad \text{for every } n \in \mathbb{Z}.
    \]
    Furthermore, one can guarantee $f \in L^2_*(I \times \mathbb{R}/\mathbb{Z})$ if and only if $ \smash{\widehat{f}} \in L^2_*(C)$. The desired conclusion then follows directly from Theorem \ref{thm:rel_mix_new}.
\end{proof}

As a direct consequence of Corollary \ref{cor:erg2} we deduce the following:

\begin{corollary}
    \label{cor:erg3}
    Let $T \colon I \to I$ be and interval exchange transformation and $$\smash{\widehat{T}} \colon I \times \mathbb{R}/\mathbb{Z} \to I \times \mathbb{R}/\mathbb{Z}$$ be an admissible affine skew product over $T$ as in \eqref{eq:skew} with $h \in \R^A_+$ a constant vector. Then, $T$ is weak mixing if and only if $\smash{\widehat{T}}$ is weak mixing.
\end{corollary}

\begin{proof}
    Corollary \ref{cor:erg2} implies eigenfunctions of $\smash{\widehat{T}}$ comes from eigenfunctions of $T$. Thus, the desired conclusion follows from the spectral characterization of weak mixing.
\end{proof}

\begin{question}
    Do Corollaries \ref{cor:erg2} and/or \ref{cor:erg3} hold for non-constant vectors $h \in \R^A_+$? What if one drops the admissibility hypothesis on the affine skew product?
\end{question}

%\francisco{I am not sure how to connect what is below with the above story...}

%\noindent\jayadev{I've written this mulitplicatively for now, but can easily change to additively.}

\subsection{The cohomological equation.} \label{sec:furst} We now show how to use an argument of Furstenberg \cite{Fu61}, inspired by work of Auslander and Green~\cite{nil2}, to prove Corollary~\ref{cor:cohomological}. We first recall \textit{Furstenberg's criterion} for general skew-products and sketch its proof. 

Let $T_0: X_0 \rightarrow X_0$ be a continuous and uniquely ergodic transformation of a compact metric space $\Om_0$ with $\mu_0$ the unique $T_0$-invariant measure on $X_0$, and let $g: X_0 \rightarrow \R/\Z$ be continuous. We then define the skew-product $T: X \rightarrow X$ on $X := X_0 \times \R/\Z$ by \begin{equation}\label{eq:skewproduct} T(x, \rho) := (T_0 (x), \rho+g(x)), \quad \text{for every } (x,\rho) \in X_0 \times \mathbb{R}/\mathbb{Z}. \end{equation} By definition, $T$ preserves the measure $\mu := \mu_0 \times m$, where $m$ is the Lebesgue probability on $\R/\Z$. In this context, \textit{Furstenberg's criterion} can be stated as follows:

\begin{theorem}
    \label{theo:furst}
    In the context above, $T$ is not uniquely ergodic if and only if there is exists $n \in \Z \setminus \{0\}$ and a measurable function $u: X_0 \rightarrow \R/\Z$ such that \begin{equation}\label{eq:furstenberg} u \circ T_0 - u = n \cdot g. \end{equation}
\end{theorem}

\begin{proof}[Sketch of proof.]
    We follow Furstenberg's argument in~\cite[Lemma 2.1]{Fu61} but translated into our additive notation. By a separate argument of Furstenberg, the unique ergodicity of $T$ is equivalent to the ergodicity of the measure $\mu$. If $T$ is not ergodic, there exists a non-constant, $T$-invariant function  $F \in L^2(X, \mu)$. Then, for any $x \in X_0$, we can consider the Fourier expansion of this function in the $\rho$ variable, i.e., \begin{equation}\label{eq:fourier} F(x, \rho) = \sum_{n \in \Z} c_n(x) e^{2\pi i n \rho}.\end{equation} Since $F$ is not constant, there exists $n_0 \in \Z\minuszero$ such that $c_{n_0}(x)$ is not identically zero. Furthemore, by $T$-invariance, we have $$\sum_{n \in \Z} c_n(T_0 x) e^{2\pi i n(\rho+g(x))} = F(T(x, \rho)) = F(x, \rho) = \sum c_n(x) e^{2\pi in \rho}.$$ By uniqueness of Fourier coefficients we deduce $$c_{n_0}(x) = c_{n_0}(T_0 x) e^{2\pi i n g(x)}.$$ Since $c_{n_0} (x) \not\equiv 0$, we can let $u(x) := \frac{1}{2\pi i} \log c_{n_0}(x)$ to obtain a function satisfying \eqref{eq:furstenberg} with $n=n_0$. Conversely, if we have a solution $u$ to \eqref{eq:furstenberg} with $n=n_0 \neq 0$, then $$F(x, \rho) := e^{2\pi i(n_0 \rho-u(x))}$$ is a $T$-invariant, non-constant function in $L^2(X,\mu)$.\\
\end{proof}

\paragraph*{\bf Applications} 
We now use Furstenberg's criterion to prove the last result of this paper.

\begin{proof}[Proof of Corollary~\ref{cor:cohomological}]
    If a solution to \eqref{eq:cohom} existed, the admissible skew-product would be non-uniquely-ergodic by Theorem \ref{theo:furst}. Thus, since Corollary \ref{cor:erg} guarantees that admissible skewing functions over uniquely ergodic interval exchange transformations give rise to uniquely erogic skew products, there are no non-trivial solutions to \eqref{eq:cohom}.\qedhere
\end{proof}

\bibliographystyle{amsalpha}


\bibliography{bibliography}

$ $

\end{document}